\documentclass{amsart}
\usepackage{cmbright}
\usepackage{enumitem}
\usepackage{amsmath, amssymb}
\usepackage[hidelinks]{hyperref}
\usepackage{doi}
\newcommand{\R}{\mathbb R}
\newcommand{\C}{\mathbb C}
\newcommand{\Z}{\mathbb Z}
\newcommand{\N}{\mathbb N}
\newcommand{\Q}{\mathbb Q}
\newcommand{\A}{\mathcal A}
\newcommand{\B}{\mathcal B}
\newcommand{\ca}{\mathcal}
\newcommand{\ideal}[1]{\ca #1}

\newcommand{\Fin}{\operatorname{Fin}}
\newcommand{\Seq}{\operatorname{Seq}}
\newcommand{\dom}{\operatorname{dom}}
\newcommand{\img}{\operatorname{Im}}

\newcommand{\nin}{\not \in}

\theoremstyle{plain}
\newtheorem{thm}{Theorem}[section]
\newtheorem{prop}[thm]{Proposition}
\newtheorem{fact}[thm]{Fact}

\newtheorem{cor}[thm]{Corollary}

\newtheorem{lem}[thm]{Lemma}

\theoremstyle{definition}
\newtheorem{rem}[thm]{Remark}
\newtheorem{defn}[thm]{Definition}

\newtheorem{exa}[thm]{Example}

\newcommand{\Jac}{\textnormal{Jac}}

\newcommand{\Div}{\textnormal{Div}}

\newcommand{\ord}{\textnormal{ord}}
\newcommand{\ev}{\textnormal{ev}}

\newcommand{\ins}{\in^*}

\newcommand{\F}{\mathcal{F}}
\newcommand{\pF}{P_{\ca F}}
\newcommand{\ext}[1]{\operatorname{ext}(#1)}
\newcommand{\PA}{\mathsf{PA}}

\newcommand{\hn}{\mathbf N}
\newcommand{\hZ}{\mathbf Z}
\newcommand{\hQ}{\mathbf Q}

\newcommand{\Var}[1]{\operatorname{Var}(#1)}
\newcommand{\nVar}[2]{\operatorname{Var}_{#2} (#1)}
\newcommand{\Points}[1]{\operatorname{Points} (#1)}
\newcommand{\pt}[1]{{#1}^{\text{points}}}
\newcommand{\Curves}[1]{\operatorname{Curves}(#1)}
\newcommand{\nCurves}[2]{\operatorname{Curves}_{#2}(#1)}
\newcommand{\affPoints}[1]{\operatorname{affPoints} (#1)}
\newcommand{\affCurves}[1]{\operatorname{affCurves}(#1)}
\newcommand{\naffCurves}[2]{\operatorname{affCurves}_{#2}(#1)}
\newcommand{\affVar}[1]{\operatorname{affVar}(#1)}
\newcommand{\naffVar}[2]{\operatorname{affVar}_{#2}(#1)}

\newcommand{\pto}{\leadsto} 
\newcommand{\pfun}[2]{[#1 \pto #2]_{\F}}

\newcommand{\fun}[2]{[#1 \to #2]_{\F}}
\newcommand{\Poly}{{\hQ[{\tt x}_i]}_{i\in \hn}^{\text{def}}}
\newcommand{\Polyn}[1]{{#1}[{\tt x}_i]_{i \in \hn}^{\text{def}}}
\newcommand{\poly}[1]{{\hQ[{\tt x}_i]}_{i<#1}^\text{def}}
\newcommand{\polyn}[2]{{#1}[{\tt x}_i]_{i<#2}^\text{def}}

\author{Alessandro Berarducci}
\address[Alessandro Berarducci]{Dipartimento di Matematica, Universit\`a di Pisa, Largo Bruno Pontecorvo 5, 56127, Pisa, Italy}
\email{alessandro.berarducci@unipi.it}

\author{Francesco Gallinaro}
\address[Francesco Gallinaro]{Centro di Ricerca Matematica ``Ennio De Giorgi'', Scuola Normale Superiore, Piazza dei Cavalieri 3, 56126, Pisa, Italy}
\email{francesco.gallinaro@sns.it}

\date{13 March 2025, Revised 22 June 2026}

\title{Projective curves and weak second-order logic} 

\begin{document}
\begin{abstract}Given an algebraically closed field $K$ of characteristic zero, we study the incidence relation between points and irreducible projective curves, or more precisely the poset of irreducible proper subvarieties  of $\mathbb P^2(K)$. Answering a question of Marcus Tressl, we prove that the poset interprets the field, and  it is in fact bi-interpretable with the two-sorted structure consisting of the field $K$ and a sort for its finite subsets. In this structure one can define the integers, so the theory is undecidable. When $K$ is the field of complex numbers we can nevertheless obtain a recursive axiomatization modulo the theory of the integers. We also show that the integers are stably embedded and that the poset of irreducible varieties over the complex numbers is not elementarily equivalent to the one over the algebraic numbers. 
\end{abstract}
\maketitle

\tableofcontents

\section{Introduction}
The logical strength of various topological lattices has been studied in \cite{Gre}. Continuing this line of research, Marcus Tressl proves, among other things, that the poset of Zariski closed proper subsets of the projective plane $\mathbb P^2(K)$ over an algebraically closed field $K$ of characteristic zero interprets the ring of integers, and therefore it has an undecidable theory \cite[7.1]{Tre}. In \cite[Section 8]{Tre} he asks whether the same holds if we restrict to irreducible varieties and whether $K$ is interpretable in the poset. We give a positive answer to both questions and we also prove that the complete theory of the poset is sensitive to the transcendence degree of the field. To this aim we use a bi-interpretability result which we discuss below. Our results also apply to the poset of reducible varieties since the irreducible ones are definable therein.  

Given a field $K$, let $\Var{K}$ be the poset of irreducible Zariski closed proper subsets of $\mathbb P^2(K)$ (points and curves) and note that $\Var{K}$ is interdefinable with the incidence relation between points and curves in $\mathbb P^2(K)$ (because there are no proper inclusions between irreducible curves). Given $n\in \N$, let $\nVar K {\leq n}$ be the substructure obtained by considering only curves of degree $\leq n$ and points. It is well-known that $\nVar K {\leq 1}$ interprets $K$ (after fixing some parameters to determine the projective frame); namely we can reconstruct the field from the incidence relation between points and lines (see \cite{Har,Hilbert}). Using this fact it is not difficult to show that, for any fixed $n$, $\nVar K {\leq n}$ is bi-interpretable with $K$ (Theorem \ref{thm:curves-field}). 

If we additionally assume that $K$ is algebraically closed of characteristic zero, a result of \cite{MD} implies that $\nVar K {\leq 2}$ is a definable subset of $\Var K$. The cited result says that a curve $C$ has degree $\leq 2$ if and only if, for any point $P$ on $C$, there is another curve that intersects $C$ only at $P$ (in the Appendix we give a proof using generalized Jacobians). Since we already know that $\nVar K {\leq 2}$ interprets $K$, it follows that $\Var K$ interprets $K$ as well. It is easy to see that the converse fails. However, we show that $\Var K$ is bi-interpretable with the two-sorted structure $(K, \Fin(K))$ obtained by adding to the field $K$ a sort for its finite subsets, together with the membership relation between $K$ and $\Fin(K)$ (Theorem \ref{thm:biinterprable}). In this structure we can define the subring of integers $\Z\subset K$, so the theory is undecidable. 

For $K$ algebraically closed, of characteristic zero, and of infinite transcendence degree, we show that the complete theory of $(K, \Fin(K))$ has the form $T_{\text{rec}} \cup T_\N$ where $T_{\text{rec}}$ is a recursively axiomatized theory and $T_\N$ is the complete theory of $(\N, +, \cdot)$ relativized to the predicate $\hn$ which defines $\N \subset K$ in $(K, \Fin(K))$ (Theorem \ref{thm:completeness}). In particular $T_{\text{rec}} \cup T_\N$ is the complete theory of $(\C, \Fin(\C))$. 

From the axiomatization it follows that a structure of the form $(K, \Fin(K))$ is elementarily equivalent to $(\C, \Fin(\C))$ if and only if $K$ is an algebraically closed field of characteristic zero and infinite transcendence degree. 
In particular, letting $\overline \Q$ be the algebraic closure of $\Q$, we have that $(\overline \Q, \Fin(\overline \Q))$ is not elementarily equivalent to $(\C, \Fin(\C))$, but we can nevertheless describe its theory (Proposition \ref{prop:Qbar}). Thanks to the bi-interpretability result we then conclude that $\Var {\overline \Q}$ is not elementarily equivalent to $\Var \C$ (Corollary \ref{cor:var}). 

Since the notion of finite set is not expressible in first-order logic, there are non-standard models $\mathcal K$ of $T_{\text{rec}}$ which are not of the form $(K, \Fin(K))$, and in which the predicate $\hn$ defines a non-standard model of Peano Arithmetic. It is important to study these models because the completeness of $T_{\text{rec}} \cup T_\N$ is obtained through a back and forth argument between saturated models with the same $\hn$. This uses a definable version of Hilbert's basis theorem for non-standard polynomials (Theorem \ref{thm:hilbert}) based on a proof of Dickson's lemma (Lemma \ref{lem:Dikson}) which can be carried out in Peano Arithmetic. We also use the existence of generic zeros for definable prime ideals of non-standard polynomials (Proposition \ref{prop:generic}). All this implies that the type over $\hn$ of any tuple of elements from a model of $T_{\text{rec}}$ is determined by a (possibly non-standard) polynomial ideal (Corollary \ref{cor:type ideal}). The development of abstract algebra for non-standard polynomials is a topic of independent interest, see for example the recent preprint \cite{MN}. Note however that the notion of non-standard polynomial in \cite{MN} is more restrictive than ours, since we are not assuming that $\hn$ is elementarily equivalent to $\N$. 

In a non-standard model $(K,\mathcal{F}(K))$ of $T_{\text{rec}}$, the sort $\mathcal{F}(K)$ is a \emph{definable finite power set} (Definition \ref{defn:hyperfinite}) of the field sort $K$, and its elements are codes for certain subsets of $K$, which we call \emph{hyperfinite}. These sets define a non-standard version of finiteness, which leads to non-standard versions of the notions of infinite transcendence degree, characteristic zero, and algebraic closedness, which are satisfied by the field sort $K$ (Definition \ref{defn:complete}).

We show that inside any such model $\mathcal K = (K, \F(K))$ we can define a model $\hn(\mathcal K) \subset K$ of first-order Peano Arithmetic and that the complete theory of $(K, \F(K))$ is determined by $T_{\text{rec}}$ and the complete theory of $\hn(\mathcal K)$
(Theorem \ref{thm:completeness-mod-peano}). When $\F(K) = \Fin(K)$, we have $\hn(\mathcal K) = \N$. Moreover, $\hn(\mathcal{K})$ is {\em stably embedded} in $(K, \mathcal{F}(K))$ in the following strong sense: any subset of $\hn(\mathcal{K})^n$ definable with parameters in $(K, \mathcal{F}(K))$ is already definable in $(\hn(\mathcal{K}), +, \cdot)$ (Theorem \ref{thm:stable-emebeddedness}). As a special case, we have that any subset of $\N^n$ definable with parameters in $(\C, \Fin(\C))$ is already definable in $(\N, +, \cdot)$, implying that there are at most countably many such subsets. By contrast, it is well-known that, if $\R$ is the real field, then in $(\R, \N)$ one can define, with parameters, every subset of $\N$. 

We finish this introduction by discussing some related results. 
Structures of the form $(K, \Fin(K))$ have been studied in \cite{Bau}. In particular \cite{Bau} shows that if $K$ is an infinite field, then $(K, \Fin(K))$ is mutually interpretable with the polynomial ring $K[x]$. Definability in polynomial rings has been studied for example in \cite[Chapter 3]{JL}, but as far as we know stable embeddedness of the integers has not been previously considered in this context.  

A natural problem is that of understanding when $K[x]$ is interpretable in $\Var K$. We do not even know whether $\Var {\R}$ interprets the field $\R$, but some related results appear in \cite{Ast, Tre}. 

Let us also note that if $K$ is the algebraic closure of a finite field then, by \cite{Wie}, every two curves in $\Var{K}$ can be swapped by an automorphism of $\Var{K}$, and therefore $\nVar{K}{\leq 1}$ is not definable. Therefore our interpretation of $K$ in $\Var{K}$ cannot be adapted to this setting.

A further problem is that of assessing the relative strength of $(K, \Fin(K))$ versus $(K, \Z)$ for various choices of $K$. We know that the multiplicative subgroup $t^\Z\subset \C$ generated by any transcendental number $t\in \C$ is not definable in $(\C, \Z)$ \cite[Proposition 2.2]{Toffalori2010}, but it is definable in $(\C, \Fin(\C))$ (Proposition \ref{prop:powers}). More generally, in $(\C, \Fin(\C))$ we can do certain recursive definitions which are not available in $(\C, \Z)$ (Theorem \ref{thm:recursion} and Proposition \ref{prop:recN}). This seems to rule out the possibility that $(\C, \Z)$ interprets $(\C, \Fin(\C))$, but we do not have a proof of this fact.

\subsection{Organisation of the paper} The paper is organised as follows. In Section \ref{sec:interpretations} we recall some model-theoretic preliminaries about Shelah's theory $T^{eq}$, and the back-and-forth method in the many-sorted setting. 
 
 In Section \ref{sec:wsostr} we study structures with a sort $\Fin(X)$ for the finite subsets of a given definable set $X$. In Sections
 \ref{sec:deffps} and \ref{sec:rec-on-sets} we introduce the notion of a definable finite power set $\F(X)$ of a definable set $X$, abstracting some properties of $\Fin(X)$ which are preserved under elementary extensions. The elements of $\F(X)$ code subsets of $X$ which we call hyperfinite. The main result of this part is a recursion theorem on hyperfinite sets (Theorem \ref{thm:recursion}). This can be interpreted as a universal property for $\F(X)$ in the definable category. 

 In Sections \ref{sec:PA0} and \ref{sec:PA} we work with models of $\PA$ which are definable in a fixed structure and satisfy a strengthening of the induction scheme. We first prove some general results assuming the existence of such a model in a given structure, and then focus on a class of fields equipped with a definable finite power set, in which these definable models of $\PA$ can be found.
 
 Sections \ref{sec:non-standard-poly} and \ref{sec:definable-ideals} study polynomial algebra in the hyperfinite setting. The results obtained there are used in Section \ref{sec:wso-complex} to give a complete axiomatization $T_{\text{rec}} \cup T_{\N}$ of $(\C, \Fin(\C))$ and prove the stable embeddedness of $(\N,+,\cdot)$ in $(\C, \Fin(\C))$. 
 
 In Section \ref{sec:Qbar} we briefly discuss the theory of $(\overline \Q, \Fin(\overline \Q))$. 
 
 In Section \ref{sec:poset} we introduce the poset of irreducible Zariski closed subsets of the projective plane, and we prove in Section \ref{sec:biinter} the bi-interpretability result. 
 
 Finally, in the Appendix, we prove the characterization of curves of degree at most 2. The result appears in \cite{MD} but it is stated there in a different setting (affine rather than projective) and depends on other papers not all of which we have been able to find.  To make the result more accessible to the reader we give a direct proof following a suggestion of Rita Pardini (which provides a proof in the smooth case).

\section{Model theoretic preliminaries}
\label{sec:interpretations}
\subsection{Shelah's theory $T^{eq}$} 
Given a first-order structure $\ca A$ (for instance a field), the family of definable sets in $\ca A$ (in the sense of first-order logic) is closed under boolean operations, cartesian products and projections on a subset of the coordinates, but in general it is not closed under taking quotients. However there is a canonical way, due to Shelah, to associate to $\ca A$ a new structure $\ca A^{eq}$ whose definable sets are stable under quotients. This is well-known (see for example \cite[Chapter 4]{Hod}) but we briefly recall some definitions to fix terminology and notation.

\begin{defn}\label{defn:eq}
Let $\ca A$ be a first-order structure, possibly many-sorted. The structure $\ca A^{eq}$ is obtained from $\ca A$ as follows: 
\begin{itemize}
\item All the sorts of $\ca A$ are also sorts of $\ca A^{eq}$. 
\item If $E \subseteq X\times X$ is a $0$-definable equivalence relation in $\ca A$, where $X$ is a product of some sorts of $\ca A$, there is a new sort $S_E$ in $\ca A^{eq}$ interpreted as the quotient $X/E$. 
\item For each $S_E$ as above, $\ca A^{eq}$ has a function symbol $p_E:X\to S_E$ interpreted as the natural projection from $X$  to $X/E$. 
\item The language $L^{eq}$ of $\ca A^{eq}$ contains, in addition to the sorts, all the functions $p_E$, and the symbols of the language $L$ of $\ca A$. The interpretation of the symbols of $L$ is the same in the two structures. 
\end{itemize} 
\end{defn}
The requirement that $E$ is $0$-definable ensures that 
if $\ca A$ and $\ca B$ are elementarily equivalent, then so are $\ca A^{eq}$ and $\ca B^{eq}$ (in particular, they have the same language). If $T$ is the complete theory of $\ca A$, the complete theory of $\ca A^{eq}$ is called $T^{eq}$. 

\begin{rem}
The family of definable sets (with parameters) of $\ca A^{eq}$ is closed under definable quotients in the following sense: if $R$ is a definable equivalence relation on a definable set $X$, then there are a definable set $Y$ in $\ca A^{eq}$ and a definable surjective map $p:X\to Y$ with kernel $R$.  

We can identify $Y$ with $X/R$ via the bijection sending $p(x)$ to the equivalence class $[x]_R$. If $p$ is understood from the context, we write $Y= X/R$ and we say that $Y$ is the quotient of $X$ with respect to $R$. Note however that $X/R$ is only determined up to a definable bijection unless we specify the projection $p$. 

If we have two such quotients with projections $p_1:X_1\to X_1/R_1$ and $p_2: X_2 \to X_2/R_2$, a function $f:X_1/R_1\to X_2/R_2$ is definable in $\ca A^{eq}$ if and only if the relation $\{(x,y) \in X_1\times X_2  \mid f(p_1(x)) = p_2(y)\}$ is definable. Similarly, a relation $f \subseteq X_1/R_1 \times X_2/R_2$ is definable if so is its preimage in $X_1\times X_2$ via $p_1\times p_2$.  
\end{rem}

\begin{exa}
Let $K$ be a field. Then the usual projection of $K^3 \setminus \{0\}$ on the projective plane $\mathbb P^2(K)$ is definable in $K^{eq}$. 
\end{exa}

\begin{defn}
Given two structures $\ca A$ and $\ca B$, possibly in different languages, we say that $\ca B$ is {\em interpretable} in $\ca A$ if $\ca B$ is isomorphic to a structure definable in $\ca A^{eq}$.
\end{defn}

A prominent example is the interpretation of the field of real numbers within a model of Euclid's axioms of geometry - or rather, their modern formulation given by Hilbert \cite{Hilbert}. Some ingredients of the proof will appear in Proposition \ref{prop:field-interpretation}.

\begin{defn}\label{defn:biinterpretable}
If two structures $\ca A$ and $\ca B$ are interpretable in each other, then by composing the interpretations we obtain an interpretation of $\ca A$ in itself and of $\ca B$ in itself. This means that there is a structure $\ca A'$ definable in $\ca A^{eq}$ which is isomorphic to $\ca A$ and a structure $\ca B'$ definable in $\ca B^{eq}$ which is isomorphic to $\ca B$. A priori the isomorphisms $\ca A' \cong \ca A$ and $\ca B' \cong \ca B$ need not be definable. However if the first one is definable in $\ca A^{eq}$ and the second one is definable in $\ca B^{eq}$, then we say that $\ca A$ and $\ca B$ are \emph{bi-interpretable}. 
\end{defn} 

\subsection{Relative elimination of quantifiers}
Let $T$ be an $L$-theory and let $\Gamma$ be a set of $L$-sentences ($L$-formulas without free variables) which is closed under boolean connectives. If $\ca M$ and $\ca N$ are models of $T$, we write $\ca M\equiv_{\Gamma} \ca N$ if $\ca M$ and $\ca N$ give the same truth value to every sentence in $\Gamma$. 
When $\Gamma$ is the set of all $L$-sentences we obtain the notion of elementary equivalence $\ca M\equiv \ca N$. If $\Gamma = \{\varphi\}$ contains a single formula $\varphi$ we write $\ca M \equiv_\varphi \ca N$. An easy application of the compactness theorem yields the following well-known result. 

\begin{prop}\label{prop:separation}
Let $\varphi$ be an $L$-sentence. Suppose that for all models $\ca M, \ca N$ of $T$, $\ca M \equiv_{\Gamma} \ca N\implies \ca M \equiv_\varphi \ca N$. Then $\varphi$ is equivalent, in $T$, to a sentence in $\Gamma$. 
\end{prop}
We need a refinement in which $\Gamma$ is allowed to have free variables, possibly restricted to a definable predicate. We use the term {\em definable predicate} as a synonym for {\em formula}, but if $T$ is a theory, a definable predicate is also understood as an equivalence class of formulas modulo provable equivalence in $T$. Given a definable predicate $X$ in $T$ and a model $\ca M$ of $T$, we write $X(\ca M)$ for the interpretation of $X$ in $\ca M$. If $\bar { x}$ is a tuple of variables of the appropriate sorts, sometimes we write $\bar { x} \in X$ instead of $X(\bar { x})$ to express the fact that $\bar { x}$ satisfies the predicate $X$.  We write $S_x$ for the sort of the variable $x$. 

\begin{prop} \label{prop:EQ} Let $\Gamma$ be a class of $L$-formulas in the free variables $\bar {x} = (x_1,\ldots, x_n)$ closed under conjunctions, disjunctions, and negation, and let $\varphi(\bar {x})$ be an $L$-formula. Let $X \subset S_{x_1}\times \ldots \times S_{x_n}$ be a definable predicate in the theory $T$. 
Suppose that for all models  $\ca M, \ca N$ of $T$ and every choice of parameters $(a_1, \ldots, a_n) \in X(\ca M)$ and $(b_1, \ldots, b_n) \in X(\ca N)$ we have
$$\ca M, a_1, \ldots, a_n \equiv_\Gamma \ca N, b_1, \ldots, b_n \implies \ca M, a_1, \ldots, a_n \equiv_\varphi \ca N, b_1, \ldots, b_n.$$
Then there is a formula $\gamma(\bar {x}) \in \Gamma$ such that $T$ proves $\forall \bar {x} \in X \; (\varphi(\bar {x}) \iff \gamma (\bar {x}))$. 
\end{prop}
\begin{proof}
Apply Proposition \ref{prop:separation} to the $(L\cup \{\bar {x}\})$-theory $T \cup \{X(\bar {x})\}$. 
\end{proof}

\subsection{Partial isomorphisms}
Given two sets $X$ and $Y$, we say that $f$ is partial function from $X$ to $Y$, written $f:X\pto Y$, if $f$ is a function from a subset of $X$ to $Y$. If $f$ is total, namely $\dom(f) = X$, we write $f:X\to Y$ with the usual arrow sign.  
We use ``map'' and ``function'' as synonyms.  

\begin{defn} \label{defn:back forth}
A family $\ca G$ of partial maps $\iota : X \pto Y$ between sets $X, Y$ has the back and forth property if:
\begin{itemize}
\item for every $\iota  \in \ca G$ and $a\in X$ there is $\eta  \in \ca G$ extending $\iota$ with $a\in \dom(\eta )$;
\item for every $\iota  \in \ca G$ and $b \in Y$, there is $\eta  \in \ca G$ extending $\iota$ with $b\in \img (\eta )$. 
\end{itemize}
\end{defn}

\begin{defn}
If $L$ is a many-sorted language and $\ca M, \ca N$ are $L$-structures, 
a partial map $\iota : \ca M \pto \ca N$ is a partial function from the union of the sorts of $\ca M$ to the union of the sorts of $\ca N$ which preserves the sorts: if $\iota(x) = y$, the sort of $x$ in $\ca M$ equals the sort of $y$ in $\ca N$. We write $a \in \ca M$ if $a$ belongs to the union of the sorts of $\ca M$. 
\end{defn}

\begin{defn}\label{defn:piso}
We say that a partial map $\iota: \ca M \pto \ca N $ between two $L$-structures is a {\em partial isomorphism} if for every $a_1, \ldots, a_n \in \dom(\iota )$ and every quantifier free formula 
$\varphi(x_1, \ldots, x_n)$, with variables appropriate for the sorts of $a_1, \ldots, a_n$, we have: 
\begin{equation*}
\ca M \models \varphi(a_1, \ldots, a_n) \iff \ca N \models \varphi (\iota (a_1), \ldots, \iota (a_n))
\end{equation*}
\end{defn}

The following result is well-known and can be easily proved by induction on the complexity of the formulas. 

\begin{prop}\label{prop:back forth}
If $\ca G$ is a family of partial isomorphisms $\iota: \ca M \pto \ca N$ with the back and forth property, then every $\iota \in \ca G$ is an \emph{elementary map}, namely the equivalence in Definition \ref{defn:piso} holds for every formula of the language, not only for the quantifier free formulas. If such a family $\ca G$ exists and is non-empty, then $\ca M$ and $\ca N$ are elementarily equivalent. 
\end{prop}

Recall that a theory $T$ is complete if and only if any two of its models are elementarily equivalent. Thus, in order to prove that a theory is complete, it suffices to show that any two models $\ca M, \ca N$ admit elementary extensions $\ca M'\succ \ca M, \ca N'\succ \ca N$ for which there is a non-empty family $\ca G$ of partial isomorphisms $\iota: \ca M' \pto \ca N'$ with the back and forth property. 

\section{Weak second-order structures}\label{sec:wsostr}

In weak second-order logic we can quantify over finite sets. This is also possible in first-order logic by adding a new sort whose elements correspond to the finite subsets of a definable set. 

\begin{defn} 
Given a structure $A$ and a definable set $X$ in $A$, let $(A, \in, \Fin (X))$ be the structure obtained from $A$ by adding a new sort $\Fin (X)$ for the finite subsets of $X$ and a relation symbol ``$\in$'' expressing the membership of elements of $X$ to elements of $\Fin (X)$. 
\end{defn}

Given a structure $A$ and a definable set $X$ in $A^{eq}$, it may happen that an isomorphic copy of $\Fin(X)$ is already definable in $A^{eq}$ in the sense of the following definition. 

\begin{defn}
We say that {\em $\Fin(X)$ is internal to $A$} if there are a definable set $Y$ of $A^{eq}$ and a definable relation $\ins$ in $A^{eq}$ such that the finite subsets of $X$ are exactly the sets of the form $\{x\in X \mid x\ins y\}$ with  $y\in Y$. In this situation $(A, \in, \Fin(X))$ is naturally isomorphic to $(A, \ins, Y)$, so our terminology is justified.
\end{defn}

Given a structure $A$ and a definable set $X$ in $A$, it is easy to see that $\Fin(X)$ is internal to $(A, \Fin(X^2))$: we can take as $Y$ the quotient of $\Fin(X^2)$ under the relation which identifies $B_1,B_2 \in \Fin(X^2)$ if and only if the set of first coordinates of elements of $B_1$ and $B_2$ coincide. More generally, if $X$ is infinite, $\Fin(X^n)$ is internal to $(A, \Fin(X^2))$ for all $n\in \N$, but not always to $(A, \Fin(X))$, as illustrated by the following proposition.

\begin{prop} \label{prop:monadic} \mbox{}
\begin{enumerate}
\item There is an infinite structure $A$ such that the theory of $(A, \Fin(A))$ is decidable \cite{Tre}. 
\item If $A$ is an infinite set, $\Fin(A^n)$ is internal to $(A, \Fin(A^2))$ for all $n\in \N$. 
\item If $A$ is an infinite set, the structure $(A, \Fin(A^2))$ interprets $(\N,+,\cdot)$, so its theory is undecidable. 
\end{enumerate} 
\end{prop}
\begin{proof}
(1) In \cite[Corollary 3.5]{Tre} the author shows, by a reduction to Rabin's results in \cite{Rabin}, that if $A$ is a dense linear order, then its weak monadic second-order theory is decidable. The setting is slightly different since \cite{Tre} includes $A$ in $\Fin(A)$ identifying a point with its singleton, but this is irrelevant.

(2) For simplicity take $n=3$. The idea is to code an element $X\in \Fin(A^3)$ by a triple $(X_1,X_2,X_3)\in \Fin(A^2)\times \Fin(A^2) \times \Fin(A^2)$ as follows. We say that $(a,b,c)\in A^3$ belongs to the set $X$ coded by $(X_1, X_2, X_3)$ if there is $t\in A$ such that $(a,t)\in X_1, (b,t)\in X_2, (c,t) \in X_3$. We introduce an equivalence relation $R$ by stipulating that two triples $(X_1,X_2,X_3)$ and $(Y_1,Y_2,Y_3)$ in  $\Fin(A^2)^3$ are $R$-equivalent if they code the same sets. 

We must show that every element of $\Fin(A^3)$ can be coded in this way. The empty set and the singletons are easily coded. Suppose that $X\subseteq A^3$ is coded by $(X_1, X_2,X_3)\in \Fin(A^2)^3$ and $u = (u_1,u_2,u_3) \in A^3$ is an arbitrary element. We must show that $X\cup \{u\}$ can be coded. To this aim we choose $t\in A$ such that $X_1 \cup X_2\cup X_3$ contains no element of the form $(x,t)$ (we use the fact that $A$ is infinite) and define $X_1 = Y_1 \cup \{(u_1,t)\}$, $Y_2 = X_2 \cup \{(u_2, t)\}$, $Y_3 = X_3 \cup \{(u_3,t)\}$. Then  $X\cup \{u\}$ is coded by $(Y_1,Y_2,Y_3)$. One can thus obtain an interpretation of $(A, \Fin(A^3))$. 

(3) By the previous point it suffices to show that $(A,\Fin(A), \Fin(A^2), \Fin(A^3))$ interprets $(\N, +, \cdot)$. Quantifying over $\Fin(A^2)$ we can express the fact that two elements of $\Fin(A)$ are equipotent (there is a bijection between the two sets) and we can define $\N$ as the quotient of $\Fin(A)$ modulo equipotence. The addition of two elements of $\N$ is defined using disjoint unions and multiplication is defined using cartesian products. The latter requires $\Fin(A^3)$ because, given $X,Y,Z\in \Fin(A)$, we must be able to say that there is a bijection $f\in \Fin(A^3)$ between $X\times Y\in \Fin(A^2)$ and $Z$. 
\end{proof}

The difference between $\Fin(A)$ and $\Fin(A^2)$ becomes irrelevant if $A$ is an infinite field because of the following lemma. 

\begin{lem} \label{lem:sets-of-pairs}Let $K$ be an infinite field. Then $\Fin(K^2)$ is internal to $(K,\Fin(K))$. 
\end{lem}

\begin{proof} 
	The idea of the proof was suggested by Marcello Mamino. Given a finite set $X \subset K^2$ consider its projections $$A = \{a \mid \exists b. (a,b) \in X\}, \quad B = \{b \mid \exists a. (a,b)\in X\}$$ to the Cartesian axes. We claim that there are $\alpha, \beta \in K$ such that the projection 
	$$
	p_{\alpha, \beta}: A\times B \to K, \quad (x,y)\mapsto \alpha x + \beta y
	$$ 
	is injective. Given this, we code $X\in \Fin(K^2)$ as the quintuple $(A,B,\alpha, \beta, C)$ where 
	$$
	C = p_{\alpha, \beta}(X) = \{\alpha x + \beta y \mid (x,y) \in X\}.
	$$
	Membership in $X$ is definable in terms of codes: $(a,b)$ belongs to the set coded by $(A,B,\alpha, \beta, C)$ if $a\in A, b\in B$ and $\alpha a + \beta b \in C$. 
	Two quintuples are equivalent if they code the same set. This equivalence relation is definable: $(A,B,\alpha', \beta', C')$ is equivalent to $(A,B, \alpha, \beta, C)$ if for all $a\in A, b\in B$ we have $\alpha a + \beta b \in C$ if and only if $\alpha' a + \beta' b \in C'$. 
Thus, we have obtained an interpretation of $\Fin (K^2)$ in $(K, \Fin(K))$. 

It remains to prove the claim. Since $A\times B$ is finite, $K^2$ contains only finitely many vectors of the form $(a-a', b-b')$, where $(a,b)$ and $(a',b')$ are distinct elements of $A\times B$. 
Since $K$ is infinite, there is a pair $(\alpha, \beta) \neq (0,0)$ such that none of these vectors lies in the kernel of $p_{\alpha,\beta}$ (for instance take $\beta = 1$ and note that we only need to exclude finitely many values of $\alpha$). The claim follows. 
\end{proof}
\begin{rem}
Under the assumptions of Lemma \ref{lem:sets-of-pairs}, a similar proof shows that $\Fin(K^n)$ is internal to $(K,\Fin(K))$. This also follows from the lemma and Proposition \ref{prop:monadic}(2). 
\end{rem}

Thanks to Lemma \ref{lem:sets-of-pairs} and Proposition \ref{prop:monadic}, if $K$ is an infinite field, then  $(K,\Fin(K))$ interprets $(\N,+,\cdot)$. If $K$ has characteristic zero, this can be strengthened by giving a definition of $\Z$ as a subring of $K$. 

\begin{lem}\label{lem:Z}
Let $K$ be a field of characteristic zero. Then the subring of integers $\Z \subseteq K$ is definable in $(K, \Fin(K))$, and so is the set of non-negative integers $\N \subset \Z$. 
\end{lem}
\begin{proof} Since $\Z = \N\cup -\N$, it suffices to define $\N$. 
Given $x\in K$, we observe that $x\in \N$ if and only if there is $F\in \Fin(K)$ such that $x \in F$ and, for all $z\in F$, we have $z\neq 0 \implies z-1 \in F$.
Indeed if $x \in \N$ we can take for $F$ the integers between $0$ and $x$ endpoints included. On the other hand if $x\nin \N$ then no such $F$ exists, because the condition forces $F$ to be closed under predecessors, which is impossible for a finite set since $K$ has characteristic zero.   
\end{proof}

Combining Lemma \ref{lem:sets-of-pairs} and Lemma \ref{lem:Z} we obtain: 

\begin{cor}
If $K$ is a field of characteristic zero, then for all $n,m \in \N$ we have that $\Fin(\Z^n \times K^m)$ is internal to $(K,\Fin(K))$.
\end{cor} 

\section{Definable finite power sets}\label{sec:deffps}
We isolate some first-order properties of $(A, \Fin(X))$ that are preserved under elementary equivalence.

\begin{defn}\label{defn:hyperfinite}
Let $\ca A$ be a first-order structure, let $X,Y$ be definable sets in $\ca A$ and let $\ins$ be a definable relation between $X$ and $Y$.  We say that $Y$ is a definable finite power set of $X$ with respect to $\ins$ if, putting $\F(X) = Y$,  the following holds:  
\begin{enumerate}
		\item (\emph{extensionality}) for every every $B,C \in \F (X)$, $B=C$ if and only if for all $x\in X$ we have $x\ins B \iff x \ins C$;
		\item (\emph{empty set}) there is $\emptyset^* \in \F (X)$ such that $x \not \ins \emptyset^*$ for each $x \in X$;
		\item (\emph{singletons}) for every $a \in X$ there is $\{a\}^* \in \F (X)$ such that for all $x\in X$ we have $x \ins \{a\}^* \iff x = a$;
		\item (\emph{binary unions}) for every $B,C \in \F (X)$ there is $B \cup^* C \in \F (X)$ such that for all $x\in X$ we have $x \ins B\cup^* C \iff x \ins B \lor x \ins C$;		
		\item (\emph{set induction scheme}) for every definable subset $\ca U \subseteq \F (X)$ if 
		$$\emptyset^* \in \ca U \quad \& \quad \forall x \in X. \forall B\in \F(X). \; (B \in \ca U \implies B \cup^* \{x\}^*\in \ca U),$$ then $\ca U = \F (X)$.
	\end{enumerate}
By extensionality the elements of $\mathcal{F}(X)$ whose existence is given by axioms (2)--(4) are uniquely determined, so the notations $\emptyset^*$ and $B \cup^* \{x\}^*$ used in the axiom scheme (5) are well-defined.
    
On the basis of the other properties, point (4) follows by induction from the special case below: 
\begin{enumerate}
\item[(4')] (\emph{union with singletons}) For every $B\in \F (X)$ and $c\in X$, there is $B \cup^* \{c\}^* \in \F (X)$ such that for all $x\in X$ we have $x \ins B\cup^* \{c\}^* \iff x \ins B \lor x =c $.
\end{enumerate}
Let us also note that (3) follows from (4') and (2), so an alternative axiomatization is given by (1),(2),(4'), (5). 
\end{defn}

\begin{prop}
If $X,Y$ and the relation $\ins$ are definable without parameters, then the fact that $Y$ is a definable finite power set of $X$ with respect to $\ins$ can be expressed by a (infinite) set of formulas without parameters. 
\end{prop}
\begin{proof}
The only point which requires a moment of thought is the set induction scheme because there we need to quantify over all parametrically definable subsets $\ca U$ of $Y$. We can handle this by considering one formula at a time and introducing a universal quantification over its parameters. 
\end{proof}

The assumption that $X,Y,\ins$ are definable without parameters is not essential since we can always reduce to that case by adding constants to the language.
\begin{cor}
An elementary extension of $\ca A = (A,\Fin(A))$ has the form $\ca A' = (A', \F(A'))$ where $\F(A')$ is a definable finite power set of $A'$ in $\ca A'$. 
\end{cor}

One of the goals of this paper is to find axioms for the complete theory of the structure $(\C, \Fin(\C))$. A naive attempt is to consider a set of axioms that express that the first sort is an algebraically closed field of characteristic zero, and the second sort is a definable finite power set of the first sort. However this is not sufficient. If it were, the theory would be decidable, but this is not the case since $(\C, \Fin(\C))$ defines $\Z$ (Lemma \ref{lem:Z}). Even adding the complete theory of $\Z$ would not suffice. 
To formulate the correct axioms, we need to develop more theory. 

\begin{defn}\label{defn:PF}
By the extensionality axiom in Definition \ref{defn:hyperfinite} an element $a\in \F(X)$ is determined by its {\em extension} 
$$\ext a = \{x\in X \mid x \ins a\} \subseteq X.$$

We denote by $\pF(X)$ the collection of extensions of elements of $\F(X)$, so we have: 
$$a \in \F(X) \iff \ext a \in \pF (X).$$ 
If $a,b\in \F(X)$, we will write $a\subseteq^* b$ for $\ext a \subseteq \ext b$.
When there is no risk of confusion we will identify elements of $\F(X)$ with their extensions and we will write $\in, \emptyset, \cup, \{x\}, \subseteq$ instead of $\ins, \emptyset^*, \cup^*, \{x\}^*, \subseteq^*$.  
\end{defn}

When it exists, a definable finite power set is determined up to a definable isomorphism. More precisely, we have: 

\begin{prop} \label{cor:up to iso} Let $X$ be a definable set in a structure $\ca A$. 
If $\F(X)$ and $\F ' (X)$ are two definable finite power sets of $X$ in $\ca A$, then $\pF(X) = P_{\F'}(X)$. 

Moreover, the map which sends an element of $\F(X)$ to an element of $\F'(X)$ with the same extension is a definable bijection.
\end{prop}
\begin{proof}
    By the scheme of set induction on $\F (X)$, for every $x\in \F(X)$ there is a unique $y\in \F'(X)$ that has the same extension of $x$. This shows that $P_{\F}(X) \subseteq P_{\F'}(X)$. The other inclusion holds by symmetry. This argument also proves the second part.
\end{proof}
We will prove the existence of many definable finite power sets under various hypotheses. The first observation is the following.
\begin{prop}
\label{rem:robust} 
Let $X\subseteq Y$ be definable sets in a structure $\ca A$ and let $\F(Y)$ be a definable finite power set of $Y$. Then the set $\F(X) = \{b \in \F(Y) \mid \ext b \subseteq X\}$ is a definable finite power set of $X$ (with $\ins_X$ being the restriction of $\ins_Y$). In particular, \[\pF(X) = P_{\ca F}(Y) \cap \ca P(X)\] where $\ca P(X)$ is the collection of all subsets of $X$.
\end{prop}

\begin{proof}
    The first part is easy and the second part follows. 
\end{proof}

\begin{defn}
A definable set $X$ is {\em hyperfinite} if $X$ has a definable finite power set $\F(X)$ and $X \in \pF(X)$.
\end{defn}

\begin{prop}\label{prop:characetrize-hyperfinite}
If $Y$ has a definable finite power set $\F(Y)$, the elements of $\pF(Y)$ are exactly the hyperfinite subsets of $Y$.
\end{prop}

\begin{proof}
    Follows from Proposition \ref{rem:robust}.
\end{proof}

\begin{rem}
If $X$ has a definable finite power set $\F (X)$, then the family of hyperfinite subsets of $X$ coincides with $\pF(X)$ by Proposition \ref{prop:characetrize-hyperfinite}, so by Definition \ref{defn:hyperfinite} it contains the empty set, the singletons of elements of $X$, and it is closed under binary unions. In particular,
$$\Fin(X) \subseteq \pF(X).$$
\end{rem}

\begin{lem} \label{lem:subset} 
A definable subset of a hyperfinite set is hyperfinite. 
\end{lem}
\begin{proof}
Let $X$ be a definable set with a definable finite power set $\F(X)$. Let $Y$ be a hyperfinite subset of $X$ and let $D$ be a definable subset of $Y$. We can write $D = \{x\in Y \mid \varphi(x,c)\}$ where $\varphi(x,y)$ is a formula and $c$ is a tuple of parameters. We prove by set induction on $Y\in \pF(X)$ that the set $\{x\in Y \mid \varphi(x,c)\}$ belongs to $\pF(X)$. This is clear if $Y$ is empty or a singleton, and the property to be proved is preserved under binary unions, so it holds for all $Y\in \pF(X)$. 
\end{proof}

\begin{prop} \label{prop:image} Fix a structure $\ca A$ and consider definability in $\ca A^{eq}$ (except for the first point, where definability in $\ca A$ suffices).  
\begin{enumerate}
\item If $f:X\to Y$ is a definable injective function and $\F(Y)$ is a definable finite power set of $Y$, then $X$ has a definable finite power set $\F(X)$. 

\item If $f:X\to Y$ is a definable surjective function and $X$ has a definable finite power set $\F(X)$, then the set $Y$ has a definable finite power set $\F(Y)$. 
 
\item If there is a definable finite power set of $\F(X\times Y)$, then there are also definable finite power sets $\F(X), \F(Y)$ of $X,Y$ respectively. 
Moreover, for every $A\in \pF(X)$ and $B\in \pF(Y)$, we have $A\times B \in \pF(X\times Y)$.  
\end{enumerate}
\end{prop}

\begin{proof} 
Point (1) is clear if $f$ is bijective. We can reduce to this case replacing $Y$ with the image of $f$ and using Proposition \ref{rem:robust}.

(2) Given a surjective definable function $f:X\to Y$ and a definable finite power set $\F(X)$ of $X$, we define an equivalence relation $\sim$ on $\F(X)$ by declaring that $A\sim B$ if the corresponding extensions $\ext A$ and $\ext B$ have the same image under $f$. In $\ca A^{eq}$ we can define $\F(Y)$ as the quotient of $\F(X)$ modulo $\sim$. We use the surjectivity assumption to show that $\F(Y)$ contains codes for the singleton sets of elements of $Y$.

(3) The projections of $X\times Y$ onto $X$ and $Y$ respectively are surjective, so the existence of $\F(X), \F(Y)$ follows from the existence of $\F(X\times Y)$. The proof of the implication $A\in \pF(X), B\in \pF(Y) \implies A\times B \in \pF(X\times Y)$ is by set induction using $A\times (B \cup \{y\}) = (A \times B) \cup (A \times \{y\})$ and the analogous equation with the roles of $A, B$ exchanged. 
\end{proof}

\begin{lem}\label{lem:projection}
     The image under a definable function of a hyperfinite set is hyperfinite.
\end{lem}

\begin{proof}
    Let $X$ be a hyperfinite set and let $f:X\to Y$ be a definable surjective function. We must show that $Y$ is hyperfinite. Since $X$ is hyperfinite, $X$ has a definable finite power set $\F(X)$, and therefore $Y$ also has a definable finite power set $\F(Y)$ by Proposition \ref{prop:image}(2). It is easy to see by the scheme of set induction on $\F(X)$ that, for every $A\in \pF(X)$, the image $\img (f_{|A})$ belongs to $\pF(Y)$. When $A = X$ we reach the desired conclusion.  
\end{proof}

\begin{lem}\label{lem:image}
Let $f:X\to Y$ be a definable surjective function. Suppose that $X$ and $Y$ have definable finite power sets $\F(X)$ and $\F(Y)$ respectively. Then for every $B \in \pF(Y)$ there is $A\in \pF(X)$ such that $B$ is the image of $A$ under $f$. \end{lem}
\begin{proof}
    By the scheme of set induction on $\F(Y)$. 
\end{proof}

In general the cartesian product of two hyperfinite sets $A,B$ need not be hyperfinite since $A\times B$ may not have a definable finite power set. However we have: 

\begin{lem}\label{lem:image-of-product}
    If $f:A\times B \to Y$ is a definable function, $A,B$ are hyperfinite, and $Y$ has a definable finite power set, then the image of $f$ is hyperfinite. 
\end{lem}
\begin{proof}
    By the scheme of set induction on $\F(B)$ we show that for every $C\in \pF(B)$ the image of $f$ restricted to $A\times C$ is hyperfinite. This is obvious if $C$ is empty. If $C = D \cup \{x\}$ with $x\nin D$, then $A \times \{x\}$ is hyperfinite by Lemma \ref{lem:projection} and by the induction hypothesis $f(A\times D)$ is hyperfinite. Thus $f(A\times C) = f(A \times D) \cup f(A \times \{x\})$ is hyperfinite. 
\end{proof}

We next show that the inclusion relation between hyperfinite sets is definably well-founded. 

\begin{thm} \label{thm:well-founded} Let $\ca A$ be a structure with a definable set $X$ which has a definable finite power set $\F (X)$. Then $(\F(X), \subseteq^*)$ is definably well-founded, namely every non-empty definable subset of $\F(X)$ contains an element whose extension is minimal with respect to the inclusion relation.
\end{thm}

\begin{proof}
Let $\ca U \subseteq \F (X)$ be a non-empty definable set. 
 We need to show that $\ca U$ has a minimal element. This would be easier if the structure $\ca A$ admitted a definable finite power set $\F (X^2)$ of $X^2$, for then we could define a notion of definable cardinality for elements of $\F(X)$ and take an element of minimal definable cardinality. 

Lacking $\F(X^2)$ we reason as follows. 
Given $Y\in \F(X)$, let $\ca U_Y = \{ B\in \F (X) \mid B \cup Y \in \ca U\}$. Now, given $B\in \F(X)$ let $B\! \downarrow = \{C \in \F(X) \mid C \subseteq B\}$. 

Let $\ca V \subset \F (X)$ be the set of all $B\in \F(X)$ satisfying the following property
\begin{equation}\label{eqn:step}
\forall Y\in \F(X). (B \! \downarrow  \cap  \; \ca U_Y\neq \emptyset \implies B \! \downarrow  \cap  \; \ca U_Y \text{ has a minimal element}).
\end{equation}
Clearly $\emptyset \in \ca V$ (because if $\emptyset \!\downarrow \cap \; \ca U_Y$ is non-empty, then it has $\emptyset$ as its sole member, and thus has $\emptyset$ as the minimal element). 

We claim that for all $B\in \F(X)$ and $x\in X$ we have
$$B\in \ca V \implies B\cup \{x\} \in \ca V$$ 
Granted the claim, we conclude as follows. Fix $B \in \ca U$. By the claim and the scheme of set induction we have $\ca V = \F(X)$, so $B\in \ca V$. Taking $Y = \emptyset$ in equation (\ref{eqn:step}) we obtain
$$B\!\downarrow \cap \; \ca U \neq \emptyset \implies B\!\downarrow \cap \; \ca U  \text{ has a minimal element }.$$
The premise of the implication holds since $B\in \ca U$. A minimal element of $B \downarrow \cap \;  \mathcal{U}$ is clearly a minimal element of $\ca U$. 

 It remains to prove the claim. We argue by set induction. Suppose $B\in \ca V$. We need to show that $B\cup \{x\}\in \ca V$. So fix $Y\in \F(X)$ such that $(B\cup \{x\})\! \downarrow \cap \; \ca U_Y  \neq \emptyset$. We need to show that $(B\cup \{x\})\! \downarrow \cap \; \ca U_Y $ has a minimal element. If $B\! \downarrow \cap \; \ca U_Y \neq \emptyset$ this follows from the hypothesis $B\in \ca V$, so we can assume that $B\! \downarrow \cap \; \ca U_Y = \emptyset$. Together with the assumption $(B\cup \{x\})\! \downarrow \cap \; \ca U_Y  \neq \emptyset$ this implies that there is a proper subset $D$ of $B$ such that $D\cup \{x\} \in \ca U_Y$. Then $D \cup \{x\} \cup Y \in \ca U$, so $D \in \ca U_{Y\cup \{x\}}$, hence $ B \! \downarrow \cap \; \ca U_{Y \cup \{x\}} \neq \emptyset$. Since $B\in \ca V$, this implies that $B \! \downarrow\cap \; \ca U_{Y \cup \{x\}}$ has a minimal element $E \subseteq B$. Then $E\cup \{x\} \in  (B\cup \{x\})\! \downarrow\cap \; \ca U_Y$ is a minimal element of $\ca U_Y$. 
 \end{proof}

The following result shows that the existence of $\F(X^2)$ is a very powerful condition. 

\begin{thm} \label{thm:F2}
Let $X$ be a definable set in a structure $\ca A$. Suppose that $X$ is not hyperfinite and that $X^2$ has a definable finite power set $\F(X^2)$. Then for all $n\in \N$ we have: 
\begin{enumerate}
\item $X^n$ has a definable finite power set $\F(X^n)$ in $\ca A^{eq}$. 
\item $\F(X^n)$ has a definable finite power set $\F(\F(X^n))$ in $\ca A^{eq}$. 
\end{enumerate}
\end{thm}
\begin{proof}
For simplicity take $n=3$. Given $Y = (Y_1, Y_2, Y_3) \in \F(X^2)\times \F(X^2) \times \F(X^2)$ we say that $Y$ codes $B \subseteq X^3$ if for all $(a_1,a_2,a_3) \in X^3$ there is at most one $t\in X$ such that $(a_1,t)\in Y_1, (a_2,t)\in Y_2, (a_3,t)\in Y_3$ and $B$ is the set of all triples $(a_1,a_2,a_3)$ for which there is such a $t$. We define an equivalence relation $R$ by declaring two elements $Y\in \F(X^2)^3$ as above $R$-equivalent if they code the same subset of $X^3$. We extend the domain of $R$ to the whole $\F(X^2)^3$ by stipulating that if $Y$ does not satisfy the above condition then $Y$ is $R$-equivalent to $(\emptyset, \emptyset, \emptyset)$, so $Y$ codes the empty set $\emptyset \subseteq X^3$. 

In $\A^{eq}$ we can now define $\F(X^3)$ as the quotient $\F(X^2)^3/R$.
Now let $P_\F(X^3)$ be the family of all sets $B\subseteq X^3$ that are coded by some $Y \in \F(X^2)^3$. To prove that $\F(X^3)$ is a definable finite power set of $X^3$ we need to prove: 
\begin{itemize}
    \item[(i)] $\emptyset \in P_{\F}(X^3)$;
    \item[(ii)] if $B \in P_{\F}(X^3)$ and $(a_1,a_2,a_3)\in X^3$ then $B \cup \{(a_1,a_2,a_3) \}\in P_\F(X^3)$;
    \item[(iii)] $\F(X^3)$ satisfies the scheme of set induction. 
\end{itemize} 
Point (i) is clear. 
To prove (ii), fix some $Y=(Y_1,Y_2,Y_3) \in \F(X^2)^3$ which codes $B$ and let $X_i \subseteq X$ be the set of all $x\in X$ such that there is $t\in X$ such that $(x,t)\in \ext {Y_i}$ ($i=1,2,3$). Then $X_i$ is hyperfinite by Lemma \ref{lem:projection}, being the projection on the first coordinate of the hyperfinite set $\ext{Y_i}$. It follows that $X_1 \cup X_2 \cup X_3 \subseteq X$ is also hyperfinite. On the other hand $X$ itself is not hyperfinite, so there is $t\in X$ such that $t\nin X_1 \cup X_2 \cup X_3$. Let $Z_i \in \F(X^2)$ be a code for the set $\ext {Y_i}\cup \{(a_i, t)\}$. Then $(Z_1, Z_2, Z_3)$ codes $B \cup \{(a_1,a_2,a_3)\}$. 

Before proving (iii), we need: 
\begin{itemize}
    \item[(*)] if $B \in P_{\F}(X^3)$ and $(a_1,a_2,a_3)\in B$ then $B \setminus \{(a_1,a_2,a_3) \}\in P_\F(X^3)$. 
\end{itemize}
To prove (*), let $t\in X$ be the unique element of $X$ such that  $(a_i,t)\in \ext{Y_i}$ for $i=1,2,3$. By Lemma \ref{lem:subset}, there is a code $Z_i\in \F(X^2)$ for the set $\ext {Y_i}\setminus \{(a_i, t)\}$. Then $(Z_1, Z_2, Z_3)$ codes $B \setminus \{(a_1,a_2,a_3)\}$. 

It remains to prove (iii). So let $\ca U\subseteq \F(X^3)$ be definable and suppose that the family of sets $B\subseteq X^3$ that are extensions of elements of $\ca U$ contains the empty set and is closed under unions with singletons. We need to prove that $\ca U = \F(X^3)$. If this is not the case, there is $Y = (Y_1, Y_2, Y_3) \in \F(X^2)^3$ such that $[Y]_R\nin \ca U$, where $[Y]_R$ is the class of $Y$ modulo $R$. Since $(\F(X^2), \subseteq^*)$ is definably well founded (Theorem \ref{thm:well-founded}), we may choose $Y$ so that $[Y]_R\nin \ca U$ and the first component $Y_1 \in \F(X^2)$ is minimal with respect to $\subseteq^*$. Let $B\subseteq X^3$ be the set coded by $Y$, namely the extension of $[Y]_R$. Since $\emptyset^* \in \ca U$, the set $B$ is non-empty. Fix a point $(a_1,a_2,a_3)\in B$. By the proof of (*), $B\setminus \{(a_1,a_2,a_3)\}$ has a code $Z = (Z_1, Z_2, Z_3)$ with $Z_i \subset^* Y_i$ for $i=1,2,3$. In particular $Z_1 \subset^* Y_1$. By the minimality assumption on $Y$, we must have $[Z]_R \in \ca U$. This implies that $[Y]_R  = [Z]_R \cup^* \{(a_1,a_2,a_2)\}^* \in \ca U$, a contradiction. 
 
\medskip
To prove (2) we define 
$$\F(\F(X^n)) = \F(X^{n+1})/R$$
where $R$ is a definable equivalence relation to be defined below. 
Given $b\in \F(X^{n+1})$ and $y\in X$, define $b_y$ as the unique element of $\F(X^n)$ such that, for all $x\in X^n$, $x\in \ext {b_y} \iff (x,y) \in \ext b$.  For $b\in \F(X^{n+1})$ we define $A_b = \{b_y \mid y\in X\} \subset \F(X^n)$. 
We can now define $R$ by putting $bRc\iff A_b = A_c$. We observe that $R$ is a definable equivalence relation on $\F(X^{n+1})$. We define $\F(\F(X^n)) = \F(X^{n+1})/R$ with the membership relation given by $x\ins [b]_R \iff x \in A_b$.  

To prove that $\F(\F(X^n))$ is a definable finite power set of $\F(X^n)$ we must verify the clauses in Definition in \ref{defn:hyperfinite}. 

Extensionality and empty set are easy. Unions with singletons can be dealt with as follows. Let $B\in \F(\F(X^n))$ and $c \in \F(X^n)$. 
We can write $B = [b]_R$ for some $b \in \F(X^{n+1})$. Since $\ext b$ is hyperfinite, its projection to the last coordinate is hyperfinite by Lemma \ref{lem:projection}. Since $X \notin \pF(X)$, there is $t \in X$ which does not lie in the projection to the last coordinate of $\ext b$. The set $\ext c \times \{t\} \subseteq X^{n+1}$ is hyperfinite by Proposition \ref{prop:image}(3), so $\ext b \cup (\ext c \times \{t\}) \subseteq X^{n+1}$ is hyperfinite, hence it has a code $q \in \mathcal{F}(X^{n+1})$ and we can define $B \cup^* \{c\}^* \in \F(\F(X^n))$ as $[q]_R$. 

It remains to verify the axiom scheme of set induction. So let $\ca U\subseteq \F(\F(X^n))$ be a definable set and suppose that $\ca U$ contains the empty set of $\F(\F(X^n))$ and is closed under taking unions with singletons in the sense of $\F(\F(X^n))$. We must prove that $\ca U = \F(\F(X^n))$, namely $[b]_R \in \ca U$ for all $b\in \F(X^{n+1})$. If this is not the case, there is a minimal $b\in \F(X^{n+1})$ with respect to $\subseteq^*$ such that $[b]_R\nin \ca U$ (Theorem \ref{thm:well-founded}). Then $[b]_R\neq \emptyset^*$, so there is $c\in \F(X^n)$ such that $c\ins [b]_R$. Note that $c\in A_b$. Let $C \subseteq \ext b$ be the subset $\{(\overline{y},z) \in \ext b  \mid b_z=c\}$. Then $\ext b \setminus C \subseteq X^{n+1}$ is a definable subset of the hyperfinite set $\ext b$, so by Lemma \ref{lem:subset} it is the extension of some $h\in \F(X^{n+1})$ with $h \subset^* b$. By the choice of $C$, we have that $[b]_R=[h]_R \cup^* \{c\}^*$. By the minimality assumption on $b$, we have $[h]_R \in \ca U$. Since $\ca U$ is closed under unions with singletons, we have $[b]_R = [h]_R \cup^* \{c\}^* \in \ca U$, contradicting the assumption on $b$. 


\end{proof}

\section{Recursion on sets}\label{sec:rec-on-sets}
In this section we consider structures which admit all possible definable finite power sets. One example is $(\C, \Fin(\C))$ (Theorem \ref{thm:field power}).   
We will show that in these structures we can do recursive definitions on the formation of hyperfinite sets (Theorem \ref{thm:recursion}). 

\begin{defn} A structure $\ca A$ \emph{admits all definable finite power sets} if every definable set $X$ in $\ca A$ has a definable finite power set $\F(X)$. 
\end{defn} 

\begin{rem}
    Let $\ca A$ be a structure. If $Y$ is a definable finite power set of $X$ in $\ca A$, then $Y$ is a definable finite power set of $X$ in $\ca A^{eq}$. If $\ca A$ admits all definable finite power sets, then so does $\ca A^{eq}$. 
\end{rem}

\begin{defn} 
We say that $(K,\F(K))$ is a {\em hyper-infinite} field if $K$ is a field, $\F(K)$ is a definable finite power set of $K$, and $K\nin \pF(K)$.
\end{defn}

\begin{lem} \label{lem:nonstandard-sets-of-pairs} Let $\ca K=(K, \F(K))$ be a hyper-infinite field. Then $K^2$ has a definable finite power set in $\ca K^{eq}$. 
\end{lem}

\begin{proof}
    The idea is to adapt the proof of Lemma \ref{lem:sets-of-pairs} to the non-standard setting. Let $A,B,C \in \F(K)$, $\alpha,\beta \in K$. We say that $(A,B,\alpha,\beta,C)$ codes a set $X \subseteq K^2$ if the following hold:
    \begin{itemize}
        \item[(i)] $\ext A$ and $\ext B$ are the projections of $X$ to the first and second coordinate respectively.
        \item[(ii)] The function $\rho_{\alpha,\beta}: \ext A \times \ext B \to K$ defined by $(x,y) \mapsto \alpha x + \beta y$ is injective.
        \item[(iii)] $\ext C= \{\rho_{\alpha,\beta}(x) \mid x \in X\}$.
    \end{itemize}
    We define an equivalence relation $R$ on the set $\ca T$ of all quintuples satisfying these conditions by saying that two quintuples are equivalent if they code the same set. We then define $\F(K^2)$ as $\ca T/R$, with the membership relation given by $(x,y) \ins [(A,B,\alpha,\beta,C)]_R \iff \rho_{\alpha,\beta}(x,y) \in \ext C$.

Clearly the empty set has a code. We must prove that if $X\subseteq K^2$ has a code $(A,B,\alpha,\beta,C)$ and $(x,y)\in K^2$, then $X \cup \{(x,y)\}$ has a code. To this aim let $A' = A \cup^* \{x\}^*$ and $B' = B\cup^* \{y\}^*$. By Lemma \ref{lem:image-of-product} the set of all differences $a-a'\in K$ with $(a,a')\in \ext {A'}^2$ is hyperfinite, and so is the set of all differences $b-b'\in K$ with $(b,b') \in \ext{B'}^2$. By another application of Lemma \ref{lem:image-of-product}, the set of elements of the form $\frac{b'-b}{a-a'}$ for $(a,b) \neq (a',b') \in \ext{A'}\times \ext{B'}$ is also hyperfinite. Since $K$ is hyper-infinite, there is $\alpha' \in K$ such that $(a-a')\alpha' \neq (b'-b)$ for all $(a,b) \neq (a',b') \in \ext{A'}\times \ext{B'}$. It follows that the function $\rho_{\alpha',1}:(a,b) \mapsto \alpha' a +b$ is injective on $\ext{A'} \times \ext{B'}$. The sets $\rho_{\alpha,\beta}(X)$ and $\rho_{\alpha',1}(X)$ are both in definable bijection with $X$, so they are in definable bijection with each other; since $\rho_{\alpha,\beta}(X)=\ext C$ is hyperfinite, so is $\rho_{\alpha',1}(X)$. Let $C' \in \F(K)$ be the element whose extension is $\rho_{\alpha',1}(X)$. Then $(A',B',\alpha',1, C' \cup^*\{\alpha' x + y \}^*)$ is a code for $X \cup \{(x,y)\}$.

To finish the proof we must verify the scheme of set induction for $\F(K^2)$. We first observe that if $X\subseteq K^2$ has a code $(A,B,\alpha,\beta,C)$ and $(x,y)\in X$, then $X \setminus \{(x,y)\}$ has a code of the form $(A',B',\alpha,\beta,C')$ with $C=C' \cup^* \{\alpha x +\beta y \}^*$ and $\alpha x  + \beta y \nin^* C'$. Let $\mathcal{U}$ be a definable subset of $\F(K^2)$ which contains the empty set of $\F(K^2)$ and is closed under unions with singletons. By definable well-foundedness of $(\F(K),\subseteq^*)$ (Theorem \ref{thm:well-founded}), if $\ca U \neq \F(K^2)$ then there is a quintuple $Q=(A,B,\alpha,\beta,C)$ with $[Q]_R \notin \ca U$ and $C$ minimal with respect to $\subseteq^*$. Removing an element $(x,y)$ from the set $X\subseteq K^2$ coded by $Q$, by the observation above we get that $X \setminus \{(x,y)\}$ is coded by a quintuple $Q'=(A',B',\alpha',\beta',C')$ with $\ext{C'} \subsetneq \ext{C}$. By minimality of $C$, we have $[Q']_R \in \ca U$, and since $\ca U$ is closed under unions with singletons we have $[Q]_R=[Q']_R \cup^*\{(x,y)\}^* \in \ca U$, a contradiction.
\end{proof}

\begin{thm}\label{thm:field power}
Let $\ca K = (K,\F(K))$ be a hyper-infinite field. Then $\ca K^{eq}$ admits all definable finite power sets. 
\end{thm} 
\begin{proof}
By Lemma \ref{lem:nonstandard-sets-of-pairs}, there is a definable finite power set of $K^2$. By Theorem \ref{thm:F2}(1), for all $n\in \N$ there is a definable finite power set $\F(K^n)$ of $K^n$. 

Every definable set $X$ in $\ca K^{eq}$ is a definable quotient of $K^m \times \F(K)^n$ for some $m,n\in \N$, so by Proposition \ref{prop:image}(2) it suffices to prove the existence of a definable finite power set of $K^m \times \F(K)^n$. There is an injective definable function $K^m \times \F(K)^n \to \F(K^{m+n})$ defined by \[(x_1,\dots,x_m,a_1,\dots,a_n) \mapsto \{x_1\}^* \times \dots \times \{x_m\}^* \times a_1 \times \dots \times a_n.\]
 By Theorem \ref{thm:F2}(2), $\F(K^{m+n})$ has a definable finite power set, hence by Proposition \ref{prop:image}(1) so does $K^m \times \F(K)^n$.
\end{proof}

\begin{lem} \label{lem:finiteness of powers}
If $X$ is hyperfinite and $\F(X)$ has a definable finite power set $\F(\F(X))$, then $\F(X)$ is hyperfinite.  
\end{lem}
\begin{proof} 
We prove by the scheme of set induction on $\F(X)$ that for all $B\in \pF(X)$, $\F(B) \in \pF(\F(X))$ where $\F(B) = \{a\in \F(X) \mid \ext a \subseteq B\}$ (this is a definable finite power set of $B$ by Proposition \ref{rem:robust}). This is easy if $B$ is empty or a singleton, for then $\F(B)$ has only one or two members. So suppose $B = C \cup \{x\}$ with $x\nin C$. By the induction hypothesis we can assume that  $\F(C) = \ext c$ for some $c \in \F(\F(X))$. A hyperfinite subset of $B$ is either contained in $C$ or is the union of $\{x\}$ with a hyperfinite subset of $C$, namely $\pF(B) = \pF (C) \cup Q$ where $Q = \{V\cup \{x\} \mid V \in \pF(C)\} \subseteq \pF(X)$. From this we deduce that $\F(B) = \F(C) \cup R$ where $R = \{x\in \F(X) \mid \ext x \in Q\} \subseteq \F(X)$. Since the union of two hyperfinite sets is hyperfinite, it suffices to show that $R$ is hyperfinite. To this aim note that there is a bijection $f_x: \F(C)\to R$ sending $u\in \F(C)$ to the unique $v \in R$ such that $\ext v = \ext u \cup \{x\}$. This bijection is the restriction to the hyperfinite subset $\F(C) \subseteq \F(X)$ of a definable function from $\F(X)$ to $\F(X)$, so by Lemma \ref{lem:subset}(2) its image $R$ is hyperfinite. 
\end{proof}

The next theorem shows that we can do recursive definitions (not only proofs by induction) on hyperfinite sets. It can be seen as a universal property for $\F(X)$ in the definable category. Essentially it says that, within the definable category, $\F(X)$ is a commutative monoid freely generated by the empty set and the singletons. We will use the theorem to define by recursion various operations on non-standard polynomials. 

\begin{thm}[Recursion on sets] \label{thm:recursion} Work in a structure admitting all definable finite power sets. 
Let $X,Y$ be definable sets and let $\F(X)$ be a definable finite power set of $X$. Given $y_0 \in Y$, a definable function $g: X \to Y$ and a commutative and associative definable function $h: Y \times Y \to Y$, there is a unique definable function 
$$G: \F(X) \to Y$$
 such that for all $A \in \F(X)$ and $x \in X$ we have: 

\begin{enumerate}
\item $G(\emptyset^*) = y_0$,
\item $G(A \cup^* \{x\}^*) = h(G(A), g(x))$,
\end{enumerate}
(where the starred operations are as in Definition \ref{defn:hyperfinite}).  Moreover, the definition of $G$ depends uniformly on the parameters of $y_0, g, h$. 
\end{thm}
\begin{proof} By Proposition \ref{rem:robust} we can assume that the definable finite power set of any $B \subseteq X$ is given by $\F(B) = \{a \in \F(X) \mid \ext a \subseteq B\}$. Therefore, if $A\subset B \subset X$ we have $\F(A) \subset \F(B) \subset \F(X)$. 
Let $\varphi(S,B,y)$ be the conjunction of:  
\begin{enumerate}
\item $B\in \F(X)$, 
\item $S$ is a definable function from $\F(B)$ to $Y$. 
\item $S(\emptyset^*) = y_0$,
\item $S(A \cup^* \{x\}^*) = h(S(A), g(x))$ \hspace{1em} $\forall A \in \F(B), \forall x \in B$,
\item $S(B) = y$. 
\end{enumerate}
Define $\gamma(B,y) :\iff \exists S. \; \varphi(S, B,y)$. We claim that $\gamma(B,y)$ is first-order definable and for all $B\in \F(X)$ there is a unique $y\in Y$ such that $\gamma(B,y)$. Granted this, we can define  $G(B) = y :\iff \gamma(B,y)$ and verify that $G$ satisfies the desired properties. 

To prove that $\gamma(B, y)$ is first-order, we observe that for any definable function $S:\F(B)\to Y$ there is a definable bijection between (the graph of) $S$ and its domain $\F(B)$. Since $B$ is hyperfinite, $\F(B)$ is hyperfinite (Lemma \ref{lem:finiteness of powers}), hence so is $S$ (Lemma \ref{lem:subset}). We have thus shown that the definable functions from $\F(B)$ to $Y$ are exactly the elements of $\pF(\F(B)\times Y)$ whose extension is a function from $\F(B)$ to $Y$. It follows that the quantifier $\exists S$ can be handled by a quantification over the elements of $\F(\F(B)\times Y)$ (we need to modify $\varphi$ so as to replace $S$ by the extension of some $s\in \F(\F(B)\times Y)$) and it is therefore first-order. 

We claim that, given $B\in \F(X)$, if there are $S,y$ such that $\varphi(S,B,y)$, then $S$ and $y$ are unique.  To this aim note that if $A\subset B$ and $\varphi(S,B,y)$ holds, then $S$ restricts to a definable function $S'$ from $\F(A)$ to $Y$ witnessing $\varphi(S',A,y')$ where $y' = S(A)$ (here we use $\F(A) \subset \F(B)$). In particular, if $B = A \cup^* \{x\}^*$, we must have $y = h(S(A), g(x)) = h(y', g(x))$ and $S(B)=y$, so the uniqueness of $S,y$ follows by set induction. 
 
It remains to prove that for all $B\in \F(X)$ there are $S,y$ such that $\varphi(S,B,y)$. For a contradiction suppose this fails. Using Theorem \ref{thm:well-founded} consider a minimal $B$ (with respect to $\subseteq^*$)  such that there are no $S,y$ satisfying $\varphi(S,B,y)$. If $B= \emptyset^*$ we define $S: {\F(\emptyset)} \to Y$ so that $S(\emptyset^*) = y_0$ and we reach a contradiction. So assume $B\neq \emptyset$. For each $x\ins B$, we can write $B = A_x \cup^* \{x\}^*$ where $x\nin^* A_x$. By the minimality of $B$ there is a unique $S_x$ and a unique $y_x$ such that $\varphi(S_x, A_x, y_x)$ holds. 
If $b,c\ins B$, we have that for every $x \ins A_{b,c} = A_b \cap^* A_c  = B\setminus^* \{b,c\}^*$ the formulas $\varphi((S_b)_{|A_{b,c}}, A_{b,c},S_b(A_{b,c}))$ and $\varphi((S_c)_{|A_{b,c}}, A_{b,c},S_c(A_{b,c}))$ both hold, so by uniqueness we have that $S_b$ and $S_c$ have a common restriction $S_{b,c}$ to $A_{b,c}$.

  Let $y_b = S_b(A_b)$, $y_c = S_c(A_c)$ and $y_{b,c} = S_{b,c}(A_{b,c})$. We must have $y_{b} = h(y_{b,c}, g(c))$ and $y_{c} = h(y_{b,c},g(b))$. Using the commutativity and associativity of $h$ it follows that $h(y_b, g(b)) = h(y_c, g(c))$. Let $y:=h(y_b,g(b))$. Since $b$ and $c$ were arbitrary, it follows that for each $x \ins B$ we have $h(y_x,g(x))=y$. Hence we can define $S:\F(B) \to Y$ as the unique common extensions of the functions $S_x$ such that $S(B)=y$. By choice of $y$, it follows that $\varphi(S, B, y)$ holds. 
\end{proof}

\section{Definable models of $\PA$}\label{sec:PA0}

In this section we work with models of $\PA$ which are definable in a given structure $\ca A$ and satisfy a strengthening of the induction scheme, which takes into account sets definable in the ambient structure. 

\begin{defn}\label{defn:PAA}
    Given a model $\ca N$ of $\PA$ definable in a structure $\ca A$, we say that $\ca N$ satisfies the induction scheme in $\ca A$, if every subset of $|\ca N|$ (the domain of $\ca N$) definable in $\ca A$ that contains zero and is closed under successor, is equal to $|\ca N|$. In this case we write $\ca N \models \PA(\ca A)$.
\end{defn}

Given a model $\ca N$ of $\PA$ and an element $n \in |\ca N|$, we will write $(n)$ for the set $\{x \in |\ca N| : x < n \}$.

\begin{lem}\label{lem:least number}
    Let $\ca A$ be a structure, and suppose $\ca N \models \PA(\ca A)$. Then every non-empty subset of $|\ca N|$ definable in $\ca A$ has a minimum. 
\end{lem}
\begin{proof}
Let $X$ be a definable subset of $|\ca N|$. 
It suffices to prove by induction on $n\in |\ca N|$ that if $X \cap (n)$ is non-empty, then it has a minimum. Both the base case and the induction step are clear. 
\end{proof}

\begin{lem}\label{lem:intervals-are-finite}
    Suppose $\ca N \models \PA(\ca A)$ and $|\ca N|$ has a definable finite power set $\F(|\ca N|)$ in $\ca A$. Then for every $n\in |\ca N|$, the interval $(n)$ is hyperfinite in the sense of $\ca A$. 
\end{lem}
\begin{proof}
    By induction on $n$, we have $(n)\in \pF(|\ca N|)$. 
\end{proof}

We prove via a G\"odel-type coding that every model $\ca N$ of $\PA$ admits all definable finite power sets. We then show that if $\ca N \models \PA(\ca A)$, then the notion of definable finite power set is absolute when moving between $\ca N$ and $\ca A$: if $X$ is a set definable in $\ca N$ and $Y$ is a definable finite power set of $X$ in $\ca N$, then it is also a definable finite power set of $X$ in $\ca A$.

\begin{lem}\label{lem:robust-pow-set-of-N}
If $\ca N \models \PA(\ca A)$, then $|\ca N|$ has a definable finite power set in $\ca N$, and if $Y$ is a definable finite power set of $|\ca N|$ in $\ca N$, then $Y$ is also a definable finite power set of $|\ca N|$ in $\ca A$.     
\end{lem}

\begin{proof}
It is well-known that the exponential function $x^y$ is definable in $\PA$. 
    For $x,y\in |\ca N|$, define $x\ins y$ if and only if there exist $a<2^x$ and $b\in |\ca N|$ such that $y = a+2^x + 2^{x+1}b$. In particular, if $\ca N$ is the standard model of $\PA$, then $x \ins y$ if and only if the $x$-th digit from the right in the binary expansion of $y$ is $1$. We claim that, in both $\ca N$ and $\ca A$, $|\ca N|$ has a definable finite power set $\F(|\ca N|)$ defined by $\F(|\ca N|)= |\ca N|$ with membership given by $\ins$. Extensionality is easy. The empty set, with respect to $\ins$, is defined by $\emptyset^* = 0$. The union with singletons is given by $x \cup^* \{y\}^* = x+2^y$ if $y\nin^* x$ and $x \cup^* \{y\}^* = x$ if $y\ins x$. 
    It remains to verify the scheme of induction. So let $\ca U \subseteq \F(|\ca N|)$ be definable in $\ca A$. Suppose that $\ca U$ contains the empty set of $\F(|\ca N|)$ and $\ca U$ is closed under unions with singletons in the sense of $\F(|\ca N|)$. We must prove that $\ca U = \F(|\ca N|)$. If this is not the case, let $x\in \F(|\ca N|) = |\ca N|$ be the least element, with respect to the order $<$ of $\ca N$, such that $x\nin \ca U$ (this exists by Lemma \ref{lem:least number}). Then $x\neq \emptyset^*$. Moreover $x$ is also minimal with respect to the inclusion $\subseteq^*$ of $\F(|\ca N|)$. We can then write $x = y\cup^* \{z\}^*$ for some $y<x$ and some $z$. By the $\subseteq^*$-minimality of $x$ we have $y\in \ca U$. By the closure properties of $\ca U$ it follows that $x\in \ca U$ and we have a contradiction. 
    
    We have thus proved the lemma in the special case when $Y = \F(|\ca N|)$. If $Y$ is another definable finite power set of $|\ca N|$ in $\ca N$, then by Proposition \ref{cor:up to iso} there is a definable bijection in $\ca N$ from $|\ca N|$ to $Y$ and we can reduce to the special case. 
\end{proof}

\begin{cor}\label{cor:pow-set-robust-in-A}
    If $\ca N \models \PA(\ca A)$, then $\ca N^{eq}$ admits all definable finite power sets, and they remain such also in $\ca A^{eq}$.     
\end{cor}
\begin{proof}
Since $\ca N$ is a model of $\PA$, for all $k\in \N$ there is a definable bijection in $\ca N$ from $|\ca N|^k$ to $|\ca N|$. By definable choice in $\ca N$, it follows that, for every definable set $X$ in $\ca N^{eq}$, there is a definable bijection in $\ca N^{eq}$ from $X$ to a definable subset of $|\ca N|$. By Lemma \ref{lem:robust-pow-set-of-N}, $|\ca N|$ has a definable finite power set $\ca F(|\ca N|)$ in both $\ca N$ and $\ca A$, unique up to a definable bijection in $\ca N$ (Proposition \ref{cor:up to iso}). Moreover, in both $\ca N$ and $\ca A$, a definable finite power set of $X\subseteq |\ca N|$, is given by $\{b\in \ca F(N) \mid \ext b \subseteq X\}$ by Proposition \ref{rem:robust}.
\end{proof}

\begin{lem} \label{lem:robust-hyperfinite}
Let $\ca N\models \PA(\ca A)$. Let $Y$ be a definable set in $\ca A^{eq}$ which is a subset of a definable set in $\ca N^{eq}$. If $Y$ is hyperfinite in the sense of $\ca A^{eq}$, then $Y$ is definable in $\ca N^{eq}$ and hyperfinite in the sense of $\ca N^{eq}$. 
\end{lem}
\begin{proof}
Let $X$ be a definable set in $\ca N^{eq}$ such that $Y \subseteq X$. 
Let $\F(X)$ be a definable finite power set of $X$ in $\ca N$, hence also in $\ca A$ by Corollary \ref{cor:pow-set-robust-in-A}.
    Since $Y\subseteq X$, a definable finite power set of $Y$ in $\ca A$ is given by $\F(Y) = \{b\in \F(X) \mid \ext b \subseteq Y\}$ by Proposition \ref{rem:robust}. Since $Y$ is hyperfinite, $Y\in \pF(Y)$, so it has the form $\ext b$ for some $b\in \F(X)$, hence it is definable in $\ca N^{eq}$ and hyperfinite in the sense of $\ca N^{eq}$.  
\end{proof}

We conclude this section by proving a result on existence of hyperfinite enumerations of hyperfinite sets which we will use later on.
 
\begin{prop}\label{prop:enumeration}
Let $\ca N \models \PA(\ca A)$. 
If $X$ is a hyperfinite set in $\ca A$ and $|\ca N| \times X$ has a definable finite power set in $\ca A$, then there exist a unique $n\in |\ca N|$ and a hyperfinite bijection $f:(n)\to X$.
\end{prop}
\begin{proof} 
By formalizing the usual proof of the pigeonhole principle, it is easy to see that there are no bijections definable in $\ca A$ between $(n)$ and $(m)$ for $n\neq m$ in $|\ca N|$, so if $n$ exists it is unique. 

Since $X$ is hyperfinite, $X$ has a definable finite power set $\F(X)$ in $\ca A$ and it is the extension of an element of $\F(X)$. By the scheme of set induction for $\F(X)$ we prove that for all $Z\in \pF(X)$, there exist $n\in |\ca N|$ and a hyperfinite relation $f\in \pF(|\ca N|\times X)$ which defines a bijection from $(n)$ to $Z$. The case $Z=\emptyset$ is obvious. For the induction step it suffices to observe that if $f:(n)\to Z$ is a hyperfinite bijection and $a\in X$ is a new element, then there is a hyperfinite bijection $g:(n+1)\to Z\cup \{a\}$ extending $f$.  
\end{proof}

\section{Defining a model of $\PA$ inside a field of definable characteristic zero}\label{sec:PA}

In this section we consider a structure of the form $\mathcal K = (K,\F(K))$ where $K$ is a field and $\F(K)$ is a definable finite power set of $K$. 

\begin{defn} \label{defn:K0} 
Let $K_0 \subseteq K$ be the set of all $x\in K$ such that there is $F\in \pF (K)$ such that $x\in F$ and, for all $z\in F$, if $z\neq 0$, then $z-1 \in F.$ 
\end{defn}

\begin{exa} Consider the standard case when $(K, \F(K))= (K, \Fin(K))$ (or more generally $\pF(K) = \Fin(K)$). 
\begin{itemize}
\item If $K$ has characteristic zero, then $K_0 = \N\subset K$ (Lemma \ref{lem:Z}). 
\item If $K$ has characteristic $p$, then $-1 \in K_0$ (because $K_0$ contains all finite subrings of $K$). 
\end{itemize}
\end{exa}

This motivates the following definition: 

\begin{defn} \label{defn:char zero} $\ca K = (K, \F(K))$ has definable characteristic zero if $-1\nin K_0$. 
\end{defn}

A field $\ca K$ of definable characteristic zero is definably infinite, so $\ca K^{eq}$ has all definable finite power sets (Theorem \ref{thm:field power}). 
In general ``definable characteristic zero'' is stronger than ``characteristic zero'', as illustrated in the following example. 

\begin{exa} \label{exa:Fp} Let $K_p$ be a field of characteristic $p$ and consider a non-principal ultraproduct $\ca K = (K, \F(K))$ of the family of structures $(K_p, \Fin(K_p))$ as $p$ varies in the primes. Then: 
\begin{itemize}
\item 
 $\ca K$ has characteristic zero, but it does not have definable characteristic zero. 
\item 
By choosing $K_p$ algebraically closed for each $p$, we can arrange so that $K \cong \C$, so we have an example of a structure of the form $(\C, \F(\C))$ which fails to have definable characteristic zero.
\end{itemize} 
\end{exa}
 
 The next result shows that there is a definable model of $\PA$ (Peano Arithmetic) inside a field of definable characteristic zero. 
 
\begin{thm}\label{thm:PA} Let $\ca K = (K, \F(K))$ be as above and let $K_0\subseteq K$ be as in Definition \ref{defn:K0}. We have: 
\begin{enumerate}
\item (zero and successor)	
$0\in K_0$ and if $y\in K_0$ then $y+1\in K_0$;
\item (induction scheme) for every definable set $\ca U\subseteq K_0$, if $$0 \in \ca U \; \land \; \forall x. (x \in  \ca U \implies x+1 \in \ca U),$$ then $\ca U = K_0$;
\item (sum and product) for all $x,y\in K_0$, $x+y\in K_0$ and $xy\in K_0$;
\item If  $\ca K$ has definable characteristic zero, then $K_0 \models \PA(\ca K)$ (see Definition \ref{defn:PAA}). 
\end{enumerate}	
\end{thm}
\begin{proof}
If $F$ and $x$ are as in Definition \ref{defn:K0} we say that $F$ witnesses $x\in K_0$. 

(1) It suffices to observe that $F = \{0\}$ witnesses $0\in K_0$ and, 
if $F$ witnesses $y\in K_0$, then $F\cup \{y+1\}$ witnesses $y+1\in K_0$. 

  (2) Let $\ca U \subseteq K_0$ be definable and assume $0\in \ca U$ and $\forall x. (x \in \ca U \implies x+1 \in \ca U)$. For a contradiction suppose there is some $x\in K_0$ with $x\nin \ca U$. By Theorem \ref{thm:well-founded} we can choose, among all possible $x$, one which has a minimal witness $F\in \pF(K)$ with respect to inclusion. Since $0\in \ca U$, we must have $x\neq 0$. Removing $x$ from $F$ we obtain a witness $G\in \pF(X)$ of $x-1 \in K_0$. By the minimality hypothesis we have $x-1\in \ca U$. This implies $x\in \ca U$ and we have a contradiction. 

(3) Consider the set $\ca U$ of all $x\in K_0$ such that for all $y\in K_0$ we have $x+y\in K_0$. Then $\ca U$ contains $0$ and is closed under successor, so by the induction scheme $\ca U = K_0$ and we have proved that $K_0$ is closed under sums. Granted this, a similar argument shows that $K_0$ is closed under products. 

(4) If $\ca K$ has definable characteristic zero, then $-1\nin K_0$, so $0$ has no predecessor in $K_0$. Together with (1)--(3) this implies that $K_0\models \PA(\ca K)$.
\end{proof}

\begin{defn} \label{defn:NZQ} Let $\ca K = (K, \F(K))$ be a field of definable characteristic zero and let $$\hn(\ca K) = K_0 \models \PA(\ca K)$$ be the model of $\PA$ introduced in Theorem \ref{thm:PA}(4). 

We define  $\hZ (\ca K) = \hn (\ca K) \cup -\hn (\ca K)$ and we let $\hQ (\ca K) \subseteq K$ be the field of all quotients $a/b$ with $a\in \hZ(\ca K)$ and $b\in \hZ (\ca K) \setminus \{0\}$. 
When $\ca K$ is clear from the context we write $\hn, \hZ, \hQ$. 
Since $\hn$ is model of $\PA$, there is a unique embedding of the non-negative integers $\N$ into $\hn$, so we can assume $\N\subseteq \hn$ and we say that $n\in \hn$ is \emph{standard} if $n\in \N$. 
\end{defn} 

Recall that $\N$ is definable in $(\Q,+,\cdot)$ by a result of Julia Robinson \cite[Theorem 3.1]{Robinson}. The same argument probably shows that $\hn(\ca K)$ is definable in the structure $(\hQ (\ca K),+,\cdot)$. For simplicity, however, we work in the structure $(\hn, \hQ, +, \cdot) = (\hQ (\ca K),\hn (\ca K),+,\cdot)$.

\begin{lem}\label{lem:robust-hyperfinite-in-Q}
    Let $X$ be a definable set in $\ca Q = (\hQ, \hn, +, \cdot)$. Then: 
    \begin{enumerate}
        \item  $\ca Q$ admits all definable finite power sets, and if $Y$ is a definable finite power set of $X$ in $\ca Q$ then $Y$ is a definable finite power set of $X$ in $\ca K$. 
   \item $X$ is hyperfinite in $\ca Q^{eq}$ if and only if $X$ is hyperfinite in $\ca K^{eq}$. 
    \end{enumerate}
\end{lem}

\begin{proof}
    Observe that in $\ca Q^{eq}$ there is a definable bijection from $X$ to a subset of $\hn$. Now reason as in Lemma \ref{lem:robust-pow-set-of-N} and Corollary \ref{cor:pow-set-robust-in-A} for point (1) and as in Lemma \ref{lem:robust-hyperfinite} for point (2). 
\end{proof}

We need to introduce some terminology for hyperfinite functions and sequences. 

\begin{defn} \label{defn:functions} Let $\ca K = (K, \F(K))$ be a field of definable characteristic zero and let $X,Y$ be definable in $\ca K^{eq}$. Let 
$$\fun X Y \;\subseteq \; \pfun X Y \; \subseteq \; \F(X\times Y)$$
be defined as follows: 
\begin{itemize}
\item $\pfun X Y$ is the set of all $f\in \F(X\times Y)$ such that $\ext f$ is a partial function from $X$ to $Y$, namely a function from a hyperfinite subset of $X$ to $Y$; 
\item $\fun X Y$ is the set of all $f\in \F(X\times Y)$ such that $\ext f$ is a total function from $X$ to $Y$. 
 \end{itemize}
 If $f\in \pfun X Y$ we sometimes identify $f$ with $\ext f$ and write $f(x) = y$ for $\ext f (x) = y$ and $\dom(f)$ for $\dom(\ext f)$. The same convention applies to the elements of $\fun X Y$. Note that if $\fun X Y$ is non-empty, then $X$ is hyperfinite (because for $f\in \fun X Y$,  $\dom(\ext f)$ is a projection of the hyperfinite set $\ext f$, so we can apply Lemma \ref{lem:subset}(2)).
 \end{defn}
 
 \begin{defn}\label{defn:sequences} Given a definable set $X$ and $n\in \hn$, we recall that $(n) = \{i\in \hn \mid i < n\}$ and we define 
 $$X^{(n)} = \fun {(n)} X, \quad \quad \Seq(X) = \bigcup_{n\in \hn} X^{(n)}$$
 So (the extension of) an element of $X^{(n)}$ is a hyperfinite sequence of length $n$ of elements of $X$ while $\Seq(X)$ is the (definable) set of all such sequences as $n$ varies in $\hn = \hn(\ca K)$. 
 \end{defn}
 
\begin{rem} Let $\ca K$ be a field of definable characteristic zero. 
\begin{itemize}
\item 
For $n$ standard there is a natural definable bijection between $X^n$ and $X^{(n)}$. 
\item For $m,n \in \hn$, there is a natural definable bijection between $X^{(m)} \times X^{(n)}$ and $X^{(m+n)}$ given by the concatenation of the two sequences. 
\end{itemize}
\end{rem}

Quantifying over hyperfinite sequences we can prove: 

\begin{prop}\label{prop:powers} Let $\ca K = (K,\F(K))^{eq}$ be a field of definable characteristic zero. There is a definable function $(n,x)\mapsto x^n$ from $\hn \times K$ to $K$ such that
\begin{itemize}
\item $x^0 = 1$. 
\item $x^{n+1} = x^n x$. 
\end{itemize}
In a similar way, given $h\in K^{(n)}$, one can define the iterated products $\prod_{m<n} h(m)\in K$. \end{prop}
\begin{proof}
We define $y = x^n$ if and only if there exists $f \in K^{(n+1)}$ satisfying $f(0) = 1$, $f(n) = y$ and $f(i+1) = x f(i)$ for all $i<n$. Such a function $f$ exists for all $n \in \hn$ by an easy induction argument. The second part is similar. 
\end{proof}

More generally, by an easier version of the proof of Theorem \ref{thm:recursion}, one can prove a recursion theorem over $\hn$: 

\begin{prop}[Recursion on numbers] \label{prop:recN} Let $\ca K$ be a field of definable characteristic zero. 
Given a definable set $Y$ in $\ca K^{eq}$, an element $y_0\in Y$ and a definable function $h:\hn \times Y \to Y$, 
there is a unique definable function $f: \hn \to Y$ satisfying
\begin{itemize}
\item $f(0) = y_0$, 
\item $f(n+1) = h(n,f(n))$.  
\end{itemize}
\end{prop}

\section{Non-standard polynomials}\label{sec:non-standard-poly}
In this section we work in a field of definable characteristic zero $\ca K = (K, \F(K))$. We define  certain rings of non-standard polynomial and an evaluation function.

\begin{defn}\label{defn:poly} Let $\hn = \hn(\ca K)$. Given $n\in \hn$ and $I\in \hn^{(n)}$, we formally identify $I$ with the {\em monomial} ${\tt x}^I = \prod_{k<n} {{\tt x}_k}^{I(k)}.$ So, $I$ is the same as ${\tt x}^I$ but we use a different notation to remind the reader of the intended interpretation of ${\tt x}^I$ as a monomial. For example, if $I = (3,4,0,2) \in \hn^{(4)}$, then ${\tt x}^I$ is the monomial ${{\tt x}_0}^3 {{\tt x}_1}^4 {{\tt x}_3}^2$.
\end{defn}

Given a definable subring $R$ of $K$, we are now ready to define the ring $\polyn R n$ of non-standard polynomials in $n$ variables with coefficients from $R$, where both $n$ and the degrees of the polynomials can be non-standard. 
 
\begin{defn}\label{defn:polynomials}
Let $R\subseteq K$ be a subring of $K$ which is definable in $\ca K = (K, \F(K))$ (for instance $R = \hQ$). Given $n\in \hn$, we define 
$$\polyn R n := \pfun {\hn^{(n)}} {R^*} \; \subseteq \; \F(\hn^{(n)} \times R^*)$$ 
where $R^* = R\setminus \{0\}$.  A (non-standard) polynomial $p\in \polyn R n$ is thus a hyperfinite function which assigns to each monomial ${\tt x}^I = I \in \hn^{(n)}$ in its domain a non-zero coefficient $p(I)\in R^*$. The domain of $p$ is also called its {\em support} and the  coefficient of a monomial which is not in the support is defined to be zero. 
With this convention, we can write $p \in \polyn R n$ as a formal expression $\sum_{I\in \hn^{(n)}} p_I {\tt x}^I$ where $p_I$ is the coefficient of ${\tt x}^I$, namely 
$$p_I = \begin{cases} p(I) & \text{ if } I \in \dom (p), \\ 0 & \text{ if } I \nin \dom (p) \end{cases}$$
We introduce a ring structure on $\polyn R n$ as follows. The sum of two polynomials $p,q \in \polyn R n$ is defined by adding the corresponding coefficients. The product $pq$ is defined by recursion on $\dom(p)$ as in the following lemma. 
\end{defn}

\begin{lem}\label{lem:product}
 There is a definable function, which, given $n\in \hn$ and $p,q\in \polyn R n$, gives a {\em product} $pq \in \polyn R n$ in such a way that the natural properties hold: the product is bilinear and commutative and the product of two monomials is obtained by adding the exponents.  
\end{lem}
\begin{proof} Given $p,q\in \polyn R n = \pfun {\hn^{(n)}} R$, the idea is to define the product $pq$ by recursion on $\dom(p)\in \F(\hn^{(n)})$ using Theorem \ref{thm:recursion}. 
More precisely, we fix $p,q,n$ as parameters and, given $A\in \F(\hn^{(n)})$, we let $p_{|A}$ be the restriction of $p$ to $\dom(p) \cap^* A$. We then define the product $p_{|A} q \in \polyn R n$ by recursion on $A\in \F(\hn^{(n)})$ as follows. 
\begin{itemize}
\item If $A$ is empty, the product $p_{|A} q$ is $0$; 
\item if $A$ is obtained from $B\subset A$ adding a new monomial ${\tt x}^I$, then $p_{|A} q = p_{|B}q + p_I {\tt x}^I q$ where $p_I$ is the coefficient of ${\tt x}^I$ in $p$ and ${\tt x}^I q$ is the polynomial which assigns to a monomial of the form ${\tt x}^{I +J}$ the coefficient $q_J$. 
\end{itemize}
\end{proof}

\begin{rem} \label{rem:standard poly} If $\pF(K) = \Fin(K)$, then $\polyn R n \cong R[x_i]_{i<n}$ where $R[x_i]_{i<n}$ is the usual ring of polynomials over $R$ in $n$ variables. 
\end{rem}

\begin{rem}
The rings $\polyn R n$ are uniformly definable as $n$ varies in $\hn$, so the union $\Polyn R = \bigcup_{n\in \hn} \polyn R n$ is definable and we have 
 \begin{align*}
 \Polyn R & = \bigcup_{n\in \hn} \pfun{\hn^{(n)}}{R^*}\\
 & \subset \; \pfun {\pfun {\hn} {\hn}} {R^*}\\
  & \subset \; \F (\F (\hn \times \hn) \times {R^*}).
 \end{align*}
 \end{rem}

\begin{defn} \label{defn:sum of polynomials}
Given $m\in \hn$ and a hyperfinite sequence $(p_i)_{i<m}$ in $({\polyn R n})^{(m)}$, we can define its hyperfinite sum $\sum_{i< m} p_i \in \polyn R n$ by recursion on $m\in \hn$ using Proposition \ref{prop:recN}. 
\end{defn}

\begin{rem}
It can be proved by induction on $m\in \hn$ that $\polyn R n$ satisfies the distributivity  law $(\sum_{i<m} p_i) q = \sum_{i<m} (p_iq)$ even for non-standard values of $m\in \hn$. 
\end{rem}

We now show how to evaluate a non-standard polynomial $p\in \polyn R n$ given an assignment $f\in K^{(n)}$ of values to its variables. 

\begin{defn} \label{defn:ev}
Given $p \in \polyn R n$ and $f\in K^{(n)}$, we want to define the value $\ev(p,f)\in K$ of $p$ at $f$ in such a way that the function $p\in \polyn R n \mapsto \ev(p,f) \in K$ is a definable morphism of rings sending ${\tt x}_i$ to $f(i)$. 

Suppose first that $p$ consists of a single monomial ${\tt x}^I$ where $I\in \hn^{(n)}$. 
Then $$\ev({\tt x}^I, f) := \prod_{m < n} f(m)^{I(m)} \in K$$ where the product is defined by recursion on $n$ (Proposition \ref{prop:powers}).
In the general case we define $p_{|A}$ as in the proof of Lemma \ref{lem:product} and we define
 $\ev(p_{|A},f)$ by recursion on $A\in \F(\hn^{(n)})$ (Theorem \ref{thm:recursion}) as follows:  
\begin{itemize}
	\item 
If $A = \emptyset^*$, then $\ev(p_{|A}, f) = 0$; 
\item if $A$ is obtained from $B\subset A$ adding a new monomial ${\tt x}^I$, then $\ev(p_{|A}, f) = \ev(p_{|B},f) + p_I \ev({\tt x}^I, f)$ where $p_I$ is the coefficient of ${\tt x}^I$ in $p$. 
\end{itemize}
\end{defn} 

\begin{rem}It is easy to see that the evaluation function respects the hyperfinite sums of polynomials introduced in Definition \ref{defn:sum of polynomials}. 
\end{rem}

We have seen that $\hQ$ is a definable subfield of $K$. With the help of the evaluation function we obtain many other definable subfields.

\begin{defn}\label{defn:generated-field}
Let $n\in \hn$ and let $f\in K^{(n)}$. We define 
$$\hQ[f]^{\text{def}} = \{\ev(p, f) \mid p\in \polyn {\hQ} n\} \subseteq K$$ 
and we let $\hQ(f)^{\text{def}} \subseteq K$ be the quotient field of the ring $\hQ[f]^{\text{def}}$. 

It is easy to see that $\hQ[f]^{\text{def}}$ depends only on the image of $f$. Moreover, for every $a\in \F(K)$ there is $n\in \hn$ and $f\in K^{(n)}$ such that $a = \img(f)$, so we can define 
$$\hQ[a]^{\text{def}} = \hQ[f]^{\text{def}}.$$ We call $\hQ[a]^{\text{def}}$ the definable subring generated by $a$ and we call $\hQ(a)^{\text{def}} = \hQ(f)^{\text{def}}$ the definable subfield generated by $a$. 

In a similar way we define the definable subring $\hQ[S]^{\text{def}} \subseteq K$ generated by a finite subset $S = \{a_1, \ldots, a_n, b_1, \ldots, b_m\}$ of $K\cup \F(K)$ where $a_1, \ldots, a_n\in K$ and $b_1, \ldots, b_m\in \F(K)$ (with $n,m$ standard): it suffices to let $A = \{a_1\}^* \cup^* \ldots, \cup^* \{a_n\}^* \cup^* b_1 \cup^* \ldots \cup^* b_m \in \F(K)$
and put $\hQ[S]^{\text{def}} = \hQ[A]^{\text{def}}$. Similarly we define $\hQ(S)^{\text{def}} = \hQ(A)^{\text{def}}$. 
\end{defn}

\begin{defn} \label{defn:algebraic} Let $R$ be a definable subring of $K$. Given $p \in R[{\tt x}]^{\text{def}}:= \polyn R 1$ and $a \in K$ we let 
$$p(a):=\ev(p, c_a)$$
 where $c_a \in K^{(1)}$ is the (constant) function with value $a\in K$.

 We say that $a\in K$ is \emph{definably algebraic} over $R$  if there is a non-zero $p \in R[{\tt x}]^{\text{def}}$ such that $p(a) = 0$.  Furthermore, if $R = \hQ[A]^{\text{def}}$ ($A\in \F(K)$) we say that $a$ is definably algebraic over $A$. 
\end{defn}

\begin{defn} \label{defn:alg closed} We say that:
\begin{itemize}
\item $\ca K$ is definably algebraically closed if for every non-constant $p\in K[{\tt x}]^{\text{def}}$, there is $a\in K$ with $p(a) = 0$. 
\item $\ca K$ has hyper-infinite transcendence degree if for every $A\in \F(K)$, there is $x\in K$ which is not definably algebraic over $A$. 
\end{itemize}
\end{defn}

\section{Definable ideals}\label{sec:definable-ideals}
As above we work in a field of definable characteristic zero $\ca K = (K, \F(K))$ and we consider definability in the sense of $\ca K$. 

\begin{defn}
Let $R$ be a definable subfield of $K$. Given $n \in \hn$ and $f\in K^{(n)}$, let 
$${\ideal I}_f = \{p\in \polyn R n \mid \ev(p,f)=0\}.$$
 Note that ${\ideal I}_f$ is a definable prime ideal and we call it the definable ideal of $f$ over $R$. 
\end{defn}

\begin{defn}
Given $n\in \hn$ and a definable ideal $\ideal I\subset \polyn R n$, we say that $\alpha\in K^{(n)}$ is a generic zero of $\ideal I$ if $\ev (p, \alpha) = 0 \iff  p\in \ideal I$ for all $p\in \polyn R n$. 
\end{defn}

Clearly if $\ideal I$ has a generic zero it is a prime ideal. 
Conversely, we have:  

\begin{prop}[Existence of generic zeros] \label{prop:generic} Let $R$ be a definable subfield of $K$ generated by a hyperfinite set (Definition \ref{defn:generated-field}). Suppose that $\ca K = (K, \F(K))$ is definably algebraically closed and of hyper-infinite transcendence degree. 
Let $n\in \hn$ and let $\ideal I \subseteq \polyn R n$ be a definable prime ideal. Then:  
\begin{enumerate}
\item 
$\ideal I$ has a generic zero. 
\item Let $r\leq n$ and let ${\ideal I}_r = \ideal I \cap \polyn R r$. Then every generic zero $\bar \alpha \in K^{(r)}$ of ${\ideal I}_r$ can be extended to a generic zero $\bar \alpha \bar \beta \in K^{(n)}$ of $\ideal I$. 
\end{enumerate}
\end{prop}
\begin{proof} Polynomial division with remainder can be extended to non-standard polynomials in one variable by induction on the non-standard degrees. For readability, we use the same notation for classical and non-standard polynomials, so we write $p(\bar {\tt x})$ for a possibly non-standard polynomial in the variables $\bar {\tt x}$ and $p(\bar \alpha)$ instead of $\ev(p, \bar \alpha)$. 

We show by induction on $r\leq n$ that ${\ideal I}_r = \ideal I \cap \polyn R r$ has a generic zero. For $r=0$ there is nothing to prove since $\ideal I_0 \cap R = \{0\}$ (because $\ideal I$ is a proper ideal). Now let $0<r<n$ and assume by induction hypothesis that ${\ideal I}_r$ has a generic zero $\bar \alpha\in K^{(r)}$. We will prove that $\ideal I_{r+1}$ has a generic zero of the form $\bar \alpha \beta\in K^{(r+1)}$ for some $\beta\in K$. 

Let $p(\bar {\tt x}, {\tt y}) \in {\ideal I}_r \subseteq \polyn R {r+1}$ where we write $\bar {\tt x}$ for the first $r$ variables and $\tt y$ for the last variable. Then $p(\bar \alpha, {\tt y})$ is a polynomial in $\tt y$ over the definable ring $R [\bar \alpha]^{\text{def}}$. 

\smallskip 
Case 1. There is a polynomial $p(\bar {\tt x}, {\tt y})\in \ideal I_{r+1}$ such that $p(\bar \alpha, {\tt y})$ is not the zero polynomial of $R(\bar \alpha)[\tt y]^{\text{def}}$. Then choose $p_{\min}(\bar {\tt x}, {\tt y}) \in \ideal I_{r+1}$ such that $p_{\min}(\bar \alpha, {\tt y})$ is not the zero polynomial and has minimal degree in $\tt y$. Choose $\beta \in K$ such that $p(\alpha, \beta) = 0$. We claim that $\bar \alpha \beta$ is a generic zero of ${\ideal I}_{r+1}$. To this aim let $s(\bar {\tt x}, {\tt y})\in R[\bar {\tt x}, \tt y]^{\text{def}}$ be an arbitrary polynomial and let us show that 
$$s(\bar \alpha, \beta) = 0 \iff s(\bar {\tt x}, {\tt y}) \in \ideal I_{r+1}.$$ 
Suppose first that $s(\bar \alpha, \beta) = 0$. Then $p_{\min}(\bar \alpha, {\tt y})$ divides $s(\bar \alpha, {\tt y})$ in $R(\bar \alpha)[\tt y]^{\text{def}}$. This implies that there are polynomials $q(\bar {\tt x}, {\tt y}), b(\bar {\tt x})$ with $b(\bar \alpha) \neq 0$ such that 
$s(\bar \alpha, {\tt y}) = p_{\min}(\bar \alpha, {\tt y}) \frac {q(\bar \alpha, {\tt y})}{b(\bar \alpha)}$.
 Let $p(\bar {\tt x}, {\tt y}) = b(\bar {\tt x}) s(\bar {\tt x}, {\tt y}) - p_{\min}(\bar {\tt x}, {\tt y}) q(\bar {\tt x}, {\tt y})$ and write it in the form $p(\bar {\tt x}, {\tt y}) = \sum_{i=0}^m c_i(\bar {\tt x}) {\tt y}^i$. By construction $p(\bar \alpha, {\tt y})$ is the zero polynomial of $R(\alpha)[\tt y]^{\text{def}}$, so $c_i(\bar \alpha) = 0$ for all $i=1, \ldots, m$.  By the induction hypothesis $c_i(\bar {\tt x}) \in {\ideal I}_r$ for all $i$, hence $p(\bar {\tt x}, {\tt y}) \in \ideal I_{r+1}$. Since also $p_{\min}(\bar {\tt x}, {\tt y}) \in \ideal I_{r+1}$, we deduce that $b(\bar {\tt x}) s(\bar {\tt x}, {\tt y}) \in \ideal I_{r+1}$. Again by induction $b(\bar {\tt x}) \nin \ideal I_{r}$ and, since ${\ideal I}_{r+1}$ is prime, $s(\bar {\tt x}, {\tt y}) \in \ideal I_{r+1}$. 

Conversely, suppose that $s(\bar {\tt x}, {\tt y}) \in \ideal I_{r+1}$. Then
$s(\bar \alpha, {\tt y}) = p_{\min}(\bar \alpha, {\tt y}) \frac {q(\bar \alpha, {\tt y})}{b(\bar \alpha)}  + \frac{r(\bar \alpha, {\tt y})}{c(\bar \alpha)}$ for some polynomials $q(\bar {\tt x},y), b(\bar {\tt x}), r(\bar {\tt x}, {\tt y}), c(\bar {\tt x})$ where $r(\bar \alpha, {\tt y})$ has lower degree in $y$ than $p_{\min}(\bar \alpha, {\tt y})$ and $b(\bar \alpha), c(\bar \alpha) \neq 0$. From the minimality of $p_{\min}(\bar {\tt x},y)$, it follows $r(\bar \alpha, {\tt y})$ is the zero polynomial, so $r(\bar \alpha, \beta) = 0$. Since $p_{\min}(\bar \alpha, \beta)=0$, we then obtain $s(\bar \alpha, \beta) = 0$. 

\smallskip
Case 2. Assume case 1 does not hold. Then for every $p(\bar {\tt x}, {\tt y})\in \ideal I_{r+1}$, the polynomial $p(\bar \alpha, {\tt y})$ is the zero polynomial of $R(\bar \alpha)[\tt y]^{\text{def}}$. Choose $\beta\in K$ transcendental over $R(\bar \alpha)^{\text{def}}$. We show that $\bar \alpha \beta$ is a generic zero of ${\ideal I}_{r+1}$. Let $s(\bar {\tt x}, {\tt y}) = \sum_{i=1}^m c_i(\bar {\tt x}) \tt y^i \in R[\bar {\tt x}, \tt y]^{\text{def}}$. 
Suppose first that $s(\bar \alpha, \beta) = 0$. Since $\beta$ is transcendental over $R(\bar \alpha)^{\text{def}}$,  $s(\bar \alpha, {\tt y}) = \sum_{i=1}^m c_i(\bar \alpha) \tt y^i$ is the zero polynomial, namely $c_i(\bar \alpha) = 0$ for all $i$. By the induction hypothesis $c_i(\bar {\tt x}) \in {\ideal I}_r$ for all $i$. Hence $s(\bar {\tt x}, {\tt y}) = \sum_{i=1}^m c_i(\bar {\tt x}) \tt y^i \in \ideal I_{r+1}$ (because definable ideals are easily shown to be closed under hyperfinite linear combinations). 

Now suppose that $s(\bar {\tt x}, {\tt y})\in \ideal I_{r+1}$. Since case 1 fails, $s(\bar \alpha, {\tt y})$ is the zero polynomial. So in particular $s(\bar \alpha, \beta) = 0$ and we are done. 
\end{proof}

The next goal is to  prove a definable version of Hilbert's basis theorem. We need a definition. 
 
\begin{defn} Let $(P,\leq)$ be a partially ordered set and let $A \subseteq P$ be a subset of $P$. 
We say that $B\subseteq A$ is a {\em basis} of $A$ if every element of $A$ is greater or equal than some element of $B$. Note that the empty set has the empty basis.
\end{defn}

\begin{rem} It is easy to see that if $A \subseteq P$ has a finite basis, then the subset of its minimal elements forms a basis, which is, in fact, the smallest basis of $A$ with respect to inclusion. However, if $A$ does not have a finite basis, the minimal elements of $A$ need not form a basis. 
\end{rem}

We need the following hyperfinite version of Dickson's lemma. Recall that $\hn = \hn(\ca K)$ is a model of $\PA(\ca K)$ (Theorem \ref{thm:PA}). Given $n \in \hn$, by the componentwise order on $\hn^{(n)}$ we mean the partial order given by $\alpha \leq \beta \leftrightarrow \alpha(i) \leq \beta(i)$ for every $i \in (n)$. 

\begin{lem}[Dickson's lemma] \label{lem:Dikson}
Given $n\in \hn$, put on $\hn^{(n)}$ the componentwise order. Then every definable subset $A\subseteq \hn^{(n)}$ has a hyperfinite basis $B\subseteq A$.  
\end{lem}
\begin{proof} 
A natural approach would be to attempt a proof by induction of the statement
$$P(n) :\iff \text{every definable subset $A$ of $\hn^{(n)}$ has a hyperfinite basis}$$ 
However, this does not work because the predicate $P(n)$ is not first-order definable, so we cannot use it in the scheme of induction when $n\in \hn$ is non-standard. To remedy, we will need to modify $P(n)$, but to motivate the changes we first prove $P(1)$ and the induction step $P(n-1)\implies P(n)$ with the above definition of $P(n)$. This will not suffices to conclude $\forall n \in \hn\; P(n)$, but then we show how to modify $P(n)$ to solve the problem. 

\smallskip
(Base case): $P(1)$ is clear because any non-empty definable subset of $\hn$ has a minimum (Lemma \ref{lem:least number}). 

\smallskip 
(Induction step):
Assume that $n>1$ and $P(n-1)$ holds. We prove $P(n)$. Let $A$ be a definable subset of $\hn^{(n)}$. For $i\in \hn$, let $A_i = \{\alpha \in \hn^{(n-1)} \mid \alpha i \in A\}$ where $\alpha i \in \hn^{(n)}$ is the concatenation of $\alpha\in \hn^{(n-1)}$ and $i\in \hn$. Note that $\bigcup_{i\in \hn}A_i$ is equal to the projection $p(A)\subseteq \hn^{(n-1)}$ of $A$ on the first $n-1$ coordinates. By $P(n-1)$ we have: 
\begin{enumerate}
\item for all $i\in \hn$, $A_i \subseteq \hn^{(n-1)}$ has a hyperfinite basis $B_i \subseteq A_i$. 
\item $p(A) \subseteq \hn^{(n-1)}$ has a hyperfinite basis $B \subseteq p(A)$. 
\end{enumerate}
Since $p(A) = \bigcup_{i\in \hn}A_i$ and $B \subseteq p(A)$ is hyperfinite, there is $j\in \hn$ such that $B \subseteq \bigcup_{i\leq j} A_i$. 
By induction on $j$ it follows that the set $\bigcup_{i\leq j} B_i \times \{i\}$ is hyperfinite. We claim that it is a basis of $A$.  To prove the claim, let $(\alpha, k) \in A$. Assume first $k \leq j$. Since $\alpha \in A_k$, there is $\beta \in B_k$ such that $\beta \leq \alpha$, and then $(\beta,k) \leq (\alpha,k)$ and $(\beta, k) \in \bigcup_{i\leq j} B_i \times \{i\}$. Assume now $k>j$. Since $\alpha \in p(A)$, there is $\beta \leq \alpha$ such that $\beta\in B$. By the choice of $j$, there is $i \leq j$ such that $\beta\in A_i$, so there is $\gamma \in B_i$ such that $\gamma \leq \beta$. Since $(\gamma, i)\leq (\alpha, k)$, the claim is proved.

\smallskip 
We conclude that $P(n)$ holds for the standard values of $n$, and we now show how to modify $P(n)$ to treat the general case.  
The idea is to replace the quantification over all definable sets in the definition of $P(n)$ with a quantification over a uniform family of definable sets which is stable under taking sections and projections (as in (1)-(2) above). These operations can be iterated in any order, so we need to handle projections on arbitrary subsets of the coordinates. This can be done as follows.

Fix a definable set $D \subset \hn \times \Seq(X)$ and for $n\in \hn$, let $D_n = \{s\in \hn^{(n)} \mid (n,s) \in D\}$. 
We want to describe a definable family $\ca C_D$ which contains $D_n$ for all $n\in \hn$ and all the sets obtained from $D_n$ by iterating the operations of taking sections and projections. We will then consider the first-order predicate
$$P_D(n) :\iff \text{every subset $A$ of $\hn^{(n)}$ belonging to $\ca C_D$ has a hyperfinite basis}$$ 
and prove $\forall n\in \hn \; P_D(n)$ as in the first part of the proof. This will suffice since every definable set $A \subset \hn^{(n)}$ belongs to a family $\ca C_D$ for some choice of $D$.

Given $t, n \in \hn$ and a function $f\in \fun {(t)} {(n)}$, let $p_f:\hn^{(n)}\to \hn^{(t)}$ be the function which sends $(x_i\mid i<n)$ to $(x_{f(i)} \mid i <t)$.  
The elements of $\ca C_D$ are sets $D_{n,m,k,\alpha,f} \subseteq \hn^{(m)}$ depending on $D$ and five parameters $n,m,k,\alpha,f$ (in addition to the parameters in $D$) with $n,m,k\in \hn$, $m+k\leq n$, $\alpha \in \hn^{(k)}$, $f\in \fun {(m+k)} {(n)}$. The definition is
$$D_{n,m,k,\alpha, f} = \{x\in \hn^{(m)} \mid \exists y \in D_n  : x\alpha = p_f(y)\}$$
where $x\alpha \in \hn^{(m+k)}$ is the concatenation of $x \in \hn^{(n)}$ and $\alpha\in \hn^{(k)}$.  
Note that $D_n = D_{n,n,0,\emptyset, \text{id}}$, so $D_n \in \ca C_D$. We claim that the family $\ca C_D$ is stable under taking sections and projections, in the sense that if $A\in \ca C_D$ and $i\in \hn$, then $A_i = \{\alpha \in \hn^{(n-1)} \mid \alpha i \in A\} \in \ca C_D$ and $\bigcup_{i\in \hn}A_i \in \ca C_D$. To this aim write $A = D_{n,m,k,\alpha,f}$ and observe that $A_i = D_{n,m-1,k+1,i\alpha, f}$ and $\bigcup_{i\in \hn} A_i = D_{n,m-1, k, \alpha, g}$ where $g(t) = f(t)$ for $t<m$ and $g(t) = f(t+1)$ for $t\geq m$.  

 Reasoning as in the first part of the proof we have $P_D(1)$ and $P_D(n-1)\implies P_D(n)$, so by the scheme of induction we have $\forall n \in \hn \; P_D(n)$. \end{proof}

\begin{defn}  
Given a definable ideal $\ideal I\subseteq \polyn R n$ and a definable subset $J$, we say that $J$ generates $\ideal I$, and write $\ideal I = \langle J \rangle$, if $\ideal I\supseteq J$ and every element of $\ideal I$ is a hyperfinite sum of multiples of elements of $J$. 
\end{defn}

Every definable subset $J$ of $\polyn R n$ generates a unique definable ideal $\ideal I$ consisting of the polynomials $p\in \polyn R n$ of the form $\sum_{i<m} g_i a_i$ where $m\in \hn$,  $(g_i)_{i<m}\in J^{(m)}$ and $(a_i)_{i<m} \in {(\polyn R n)}^{(m)}$. 

\begin{defn}
A definable ideal $\ideal I \subseteq \polyn R n$ is called a monomial ideal if it is generated by monomials. 
\end{defn}

Note that if $\ideal I$ is a monomial ideal, then the monomials of any element $f$ of $\ideal I$ belong to $\ideal I$. Indeed, assume $f=\sum_i h_i m_i$ where the $m_i$ are monomials in $\ideal I$. Each summand $h_im_i$ can be written as a $R$-linear combination of monomials in $\ideal I$. Now use the fact that if two polynomials coincide they have the same monomials.

\begin{cor}[Rephrasing of Dickson's lemma] \label{cor:Dikson2}
Every definable monomial ideal $\ideal I$ in $\polyn R n$ is hyperfinitely generated by monomials, namely there is a hyperfinite set of monomials $J\in \pF(\hn^{(n)})$ such that, for each $f \in  \ideal I$, every monomial of $f$ is divisible by some monomial in $J$.   
\end{cor}

\begin{proof} 
Let $D\subseteq \hn^{(n)}$ be the set of monomials in $\ideal I$ and note that  $D$ is definable and $\ideal I = \langle D \rangle$. By Lemma \ref{lem:Dikson}, $D$ has a hyperfinite basis $J\in \pF (\hn^{(n)})$. Since the partial order on $\hn^{(n)}$ corresponds to the divisibility of monomials, $D \subseteq \langle J \rangle$, so $\ideal I = \langle J \rangle$.  
\end{proof}

We can deduce Hilbert's basis theorem from Dickson's lemma as in the classical case (see \cite[Theorem 1.13]{Stu}). 

\begin{thm}[Hilbert's basis theorem] \label{thm:hilbert} 
Every definable ideal $\ideal I$ in $\polyn R n$ is hyperfinitely generated. 
\end{thm}

\begin{proof}
Fix a definable total order $\prec$ on the set of monomials which refines the componentwise order, for instance the lexicographic order. Every polynomial $p$ has a hyperfinite set of monomials, so by induction it has a largest monomial $\text{in}_\prec (p)$ with respect to $\prec$. We call $\text{in}_\prec (p)$ the initial monomial of $p$. Now let $\text{in}_\prec (\ideal I)$ be the ideal generated by the initial monomials of the elements of $\ideal I$.
By Corollary \ref{cor:Dikson2} $\text{in}_\prec (\ideal I)$ can be generated by a hyperfinite subset $J \subset \text{in}_\prec (\ideal I)$ consisting of monomials. Each element of $J$ is the initial monomial of some element of $\ideal I$.  
Since $J$ is hyperfinite, by Lemma \ref{lem:image} there is a hyperfinite set $G\subseteq \ideal I$ such that $\{\text{in}_\prec (p)\mid p\in G\} = J$. 
We claim that $G$ generates $\ideal I$. If this is not the case, there exists $p\in \ideal I$ which is not in the ideal generated by $G$ and whose initial monomial $\text{in}_\prec (p)$ is minimal with respect to the ordering $\prec$. Since $\text{in}_\prec(p) \in \text{in}_\prec(\ideal I)$, there are $g\in G$ and a monomial ${\tt x}^d$ (with  $d\in \hn^{(n)}$) such that $p-{\tt x}^d g$ is a polynomial with a strictly smaller initial monomial. But then $p - {\tt x}^d g$ and ${\tt x}^d g$ are generated by $G$, hence so is $p$.  
\end{proof}

Let $\hn = \hn(\ca K)$, $\hQ= \hQ(\ca K)$, and consider the structure $\ca Q = (\hQ, \hn, +, \cdot)$. Then $\ca Q$ admits all definable finite power sets, and the definable finite power sets in the sense of $\ca Q$ are also definable finite power sets in the sense of $\ca K^{eq}$ (Lemma \ref{lem:robust-hyperfinite-in-Q}(1)). By Definition \ref{defn:polynomials}, $\poly n  \subseteq \; \F(\hn^{(n)} \times \hQ^*)$, so we may assume that the ring of non-standard polynomials $\poly n$ is a definable set in $(\hQ,\hn,+,\cdot)^{eq}$.

\begin{cor} \label{cor:ideal}  
Every ideal $\ideal I \subseteq \poly n$ definable in $(K, \F(K))^{eq}$ is definable in $(\hQ, \hn,+,\cdot)^{eq}$. 
\end{cor} 

\begin{proof}
By Theorem \ref{thm:hilbert}, $\ideal I$ has a hyperfinite basis $G\in \F(\poly n)$. Then $G$ is definable in $(\hQ, \hn, +, \cdot)^{eq}$ by Lemma \ref{lem:robust-hyperfinite-in-Q}(2). By Proposition \ref{prop:enumeration} there is a hyperfinite enumeration $(g_i)_{i<m} \in  (\poly n)^{(m)}$ of $G$. Then $p\in \ideal I$ if and only if there is a hyperfinite sequence $(a_i)_{ i<m}\in (\poly n)^{(m)}$ such that $p = \sum_{i<m}a_i g_i$ (defined by recursion using Proposition \ref{prop:recN} in $(\hQ, \hn, +,\cdot)^{eq}$; one can also check that the polynomial operations do not depend on the ambient structure). This gives a definition of $\ideal I$ in $(\hQ, \hn,+,\cdot)^{eq}$. 
\end{proof}

By a similar argument, one can prove more generally that if $R$ is a subfield of $K$ definable in $(K,\F(K))$, then every definable ideal $\ideal I\subseteq \polyn R n$ is definable with parameters from $\F(R)$.
\section{The weak second-order theory of the complex field}\label{sec:wso-complex}

In this section we describe the complete theory $T^{\Fin}_\C$ of the structure $(\mathbb{C},\Fin(\mathbb{C}))$ and a recursive subtheory $T_{\text{rec}} \subset T^{\Fin}_\C$. 

\begin{defn}\label{defn:complete} Let $T_{\text{rec}}$ be the recursive theory whose models are the structures $\ca K = (K,\F(K))$ satisfying the following properties: 
\begin{enumerate}
\item $K$ is a field. 
\item $\F (K)$ is a definable finite power set of $K$ (Definition \ref{defn:hyperfinite}). 
\item $\ca K$ has definable characteristic zero (Definition \ref{defn:char zero}). 
\item $\ca K$ is definably algebraically closed (Definition \ref{defn:alg closed}). 
\item $\ca K$ has hyper-infinite transcendence degree (Definition \ref{defn:alg closed}). 
\end{enumerate}
\end{defn}
By Theorem \ref{thm:PA}, if $\ca K\models T_{\text{rec}}$, then $(\hn(\ca K),+,\cdot)$ is a model of $\PA$. By G\"odel's incompleteness theorems, since $T_{\text{rec}}$ is recursive, $T_{\text{rec}}$ cannot be complete. 

\begin{defn} 
Let $T^{\Fin}_\C = T_{\text{rec}} \cup T_\N$ be the union of $T_{\text{rec}}$ and an axiom scheme $T_\N = \{\varphi^{\hn} \mid (\N,+,\cdot) \models \varphi\}$ expressing the fact that every true formula of arithmetic holds when relativized to the definable predicate $\hn$. 

The models of $T^{\Fin}_\C$ are the models $(K, \F(K))$ of $T_{\text{rec}}$ such that $(\hn(\ca K),+,\cdot) \equiv (\mathbb N, +, \cdot)$.
\end{defn}

We will prove that the complete theory of $(\C, \Fin(\C))$ is axiomatized by $T^{\Fin}_\C$. 
Clearly $(\C, \Fin(\C))$ is a model of $T^{\Fin}_\C$. To prove the completeness of $T^{\Fin}_\C$, it suffices to show that any two models $\ca A, \ca B$ of $T^{\Fin}_\C$ are elementarily equivalent. We need the following lemma. 

\begin{lem} \label{lem:main} Let $\mathcal A= (A, \F (A))$ and $\mathcal B = (B, \F(B))$ be models of $T_{\text{rec}}$. Suppose that $\Psi: \hn(\ca A) \cong \hn (\ca B)$ is an isomorphism in the language $\{+,\cdot\}$. Then $\Psi$ is an elementary map when considered as a partial function from $\ca A$ to $\ca B$. In particular $\ca A \equiv \ca B$. 
\end{lem} 

\begin{proof} It suffices to show that the restriction of $\Psi$ to any finite subset of $\hn (\ca A)$ belongs to a family of partial isomorphisms $\ca G$ from $\ca A$ to $\ca B$
 with the back and forth property. Replacing $\ca A$ with an isomorphic structure, we may assume that $\hn(\ca A) = \hn(\ca B)$, and $\Psi$ is the identity on $(\hn,+,\cdot) = (\hn (\ca A),+,\cdot) = (\hn (\ca B),+,\cdot)$. 
We can then assume that $\hQ(\ca A) = \hQ(\ca B)$ and 
write
$$(\hQ, \hn, +, \cdot)^{eq} = (\hQ(\ca A), \hn (\A), +,\cdot)^{eq} = (\hQ(\ca B), \hn (\B), +,\cdot)^{eq}$$
In particular, since $\Poly$ and its operations are definable in $(\hQ, \hn, +,\cdot)^{eq}$, we may assume that $\ca A$ and $\ca B$ have the same non-standard polynomials over $\hQ$. 

The idea is to show that the relations between elements of $A\cup \F(A)$ can be coded by definable ideals in $\poly n$ (for various $n$) and such ideals are definable in $(\hQ,\hn, +, \cdot)^{eq}$ (Corollary \ref{cor:ideal}). We will use such ideals to define a family $\ca G$ of partial isomorphisms with the back and forth property. 
To this aim we observe that an element $a \in \F(A)$ can be coded by a function $f\in A^{(n)}$ with $a = \img(f)$ (Proposition \ref{prop:enumeration}) where $n\in \hn$ can be non-standard. Thus we need to ensure that the relations between elements of   $\bigcup_{n \in \hn} A^{(n)}$ be preserved by the partial maps in the family. This motivates the following definition. 

Let $\ca G_0$ be the set of all partial maps $\iota : \bigcup_{n \in \hn} A^{(n)} \pto \bigcup_{n\in \hn} B^{(n)}$ with the following properties: 
\begin{itemize}
\item $\dom(\iota)$ is finite (in the standard sense). 
\item If $n\in \hn$ and $f\in A^{(n)} \cap \dom(\iota)$ and then $\iota(f) \in B^{(n)}$. 
\item If $\{f_1, \ldots, f_k\} \subseteq \dom(\iota)$ and $f = f_1*\ldots *f_k \in A^{(m)}$ (where ``$*$'' is the concatenation), the definable ideal ${\ideal I}_f \subseteq \poly m$ of $f$ in $\ca A$ coincides with the definable ideal ${\ideal I}_g$ of $g = \iota(f_1)*\ldots * \iota(f_k)$ in $\ca B$. 
\end{itemize}

We claim that $\ca G_0$ has the back and forth property. 
Given $\iota$ as above with $\dom(\iota) = \{f_1, \ldots, f_k\}$ and $\img(\iota) = \{g_1, \ldots, g_k\}$, consider a new $f$ and we need to find a corresponding $g$. Let $\ideal I$ be the definable ideal of  $f_1*\ldots f_k*f$ in $\ca A$. Then $\ideal I$ is definable in $(\hQ,\hn, +,\cdot)^{eq}$ (Corollary \ref{cor:ideal}), hence it is also definable in $\ca B$ (by the same formula) and therefore it has a generic zero of the form $g_1* \ldots * g_k*g$ (Proposition \ref{prop:generic}(2)). This proves the forth direction and the back direction is analogous. 

\bigskip 
Given $\iota \in \ca G_0$, let $\hat \iota: A \cup \F(A) \pto B \cup \F(B)$ be the smallest function such that, for all $f\in A^{(n)}$ and $m<n$, if $\iota(f) = g$, then $\hat \iota$ maps $f(m)\in A$ to $g(m)\in B$ and $\img (f) \in \F(A)$ to $\img(g) \in \F(B)$. 
To verify that $\hat \iota$ is well defined, suppose that $f_1\in A^{(n_1)}$ and $f_2\in A^{(n_2)}$ belong to $\dom(\iota)$ and $f_1(m_1) = f_2(m_2)$ (where $f_1,f_2$ are not necessarily distinct). Then the definable ideal of $f_1*f_2$ contains the polynomial $\tt x_{m_1} - \tt x_{n_1+m_2}$. This polynomial then lies in the definable ideal of $g_1*g_2$ where $g_1 = \iota(f_1)$ and $g_2 = \iota(f_2)$. Since $\iota$ preserves the information contained in the ideal, $g_1(m_1) = g_2(m_2)$. Similarly one shows that if $\img(f_1) = \img(f_2)$, then $\img(g_1) = \img(g_2)$. The same arguments show that $\hat \iota$ is injective. Let us also observe that if $x\in \hQ$ belongs to the domain of $\hat \iota$, then $\hat \iota (x) = x$.

Now let $\ca G = \{\hat \iota \mid \iota \in \ca G_0\}$. Then $\ca G$ is a family of injective partial maps from $A\cup \F(A)$ to $B \cup \F(B)$ and it has the back and forth property because $\ca G_0$ does. 

We claim that every element $\hat \iota$ of $\ca G$ is a partial isomorphism, namely it preserves the quantifier free formulas, or equivalently the atomic formulas (negations can be handled by considering the inverse of $\hat \iota$). So we must prove: 
$$
\begin{cases}
x\in a \implies \hat \iota(x) \in \hat \iota(a)\\
x+y = z \implies \hat \iota(x) + \hat \iota(y) = \hat \iota(z)\\
x\cdot y= z \implies \hat \iota(x)\cdot \hat \iota(y) = \hat \iota(z)
\end{cases}
$$
Suppose for instance that $a\in \dom(\hat \iota)$ and $x\in a$. Then there is $f\in \dom(\iota)$ such that $a = \img(f)$ and there is $m\in \hn$ such that $x = f(m)$. Let $g = \iota(f)$. Then by definition of $\hat \iota$ we have $\hat \iota (x) = g(m) \in \img (g) = \hat \iota (\img f)) = \hat \iota (a)$. 

Now suppose that $x,y,z\in \dom(\hat \iota)$ and $x+y= z$. Then we can write $x = f_1(m_1), y=f_2(m_2), z= f_3(m_3)$ where $f_1 \in A^{(n_1)}$, $f_2\in A^{(n_2)}$, $f_3 \in A^{(n_3)}$ are in $ \dom(\iota)$ (not necessarily distinct) and $m_i<n_i$ for $i=1,2,3$. Let $g_i = \iota(f_i)$. Then the polynomial $\tt x_{m_1} + \tt x_{n_1+m_2} - \tt x_{n_1+n_2+m_3}$ belongs to the definable ideal of $f_1*f_2*f_3$, hence also to the definable ideal of $g_1*g_2*g_3$. It follows that $\hat \iota (x) + \hat \iota(y) = g_1(m_1) + g_2(m_2) = g_3(m_3) =  \hat \iota (z)$. 
The proof of $x\cdot y= z \implies \hat \iota(x)\cdot \hat \iota(y) = \hat \iota(z)$ is analogous. 

Since for any $n\in \hn \cap \dom(\hat \iota)$, we have $\hat \iota (n) = n = \Psi(n)$, it follows that every finite restriction of $\Psi$ can be extended to a map in $\ca G$ and therefore $\Psi$ is an elementary map by Proposition \ref{prop:back forth}. 
\end{proof}

\begin{cor}\label{cor:iota elementary} Let $\mathcal A= (A, \F (A))$ and $\mathcal B = (B, \F(B))$ be models of $T_{\text{rec}}$. Suppose that $\Psi: \hn(\ca A) \cong \hn (\ca B)$ is an isomorphism in the language $\{+,\cdot\}$. Let $\ca G_0$ be defined as in the proof of Lemma \ref{lem:main}. Then any $\iota \in \ca G_0$ is a partial elementary map from $(A, \F(A))^{eq}$ to $(B, \F(B))^{eq}$. 
\end{cor}

\begin{proof} 
We may assume $\hn (\ca A) = \hn(\ca B) = \hn$ and $\Psi$ is the identity. Recall that $\ca G_0$ is a set of partial maps $\iota : \bigcup_{n \in \hn} A^{(n)} \pto \bigcup_{n\in \hn} B^{(n)}$ with the back and forth property. 
This is not yet sufficient to conclude that every map in $\ca G_0$ is elementary since in the definition of the back and forth property (Definition \ref{defn:back forth}) we have not included the requirement that atomic formulas be preserved. However in the proof of Lemma \ref{lem:main} we have defined a family $\ca G$ of partial isomorphism from $A\cup \F(A)$ to $B \cup \F(B)$ with the back and forth property (hence a family of partial elementary maps), and we have seen that to any partial map $\iota\in \ca G_0$ we can associate a map $\hat \iota \in \ca G$. By definition, for every $f\in A^{(n)}\cap \dom(\iota)$, the domain of $\hat \iota$ contains the following elements: the element $\img(f)\in \F(A)$ whose extension is the image of the extension of $f$; the elements $f(m) \in A$ for $m<n$ (where $f(m) = y \iff (m,y) \in \ext f$). Moreover
$$
\begin{cases}
\hat \iota (\img(f)) = \img (\iota(f))\\
\hat \iota (f(m)) = \iota (f) (m) \text{ for } m<n
\end{cases}
$$
Every $\hat \iota \in \ca G$ is a partial elementary map, hence it induces a partial elementary map $\iota^{eq}$ from $\ca A^{eq}$ to $\ca B^{eq}$ in the natural way. We let $\ca G^{eq} = \{\gamma^{eq} \mid \gamma \in \ca G\}$. To prove that $\iota \in \ca G_0$ is an elementary map it suffices to show that $\iota$ can be extended to a map $\zeta = \eta^{eq}$ in $\ca G^{eq}$. Some work is needed, since there is no guarantee that $\iota^{eq}$ itself extends $\iota$ (the domain of $\iota^{eq}$ may not include the domain of $\iota$). 

Recall that $\bigcup_{n\in \hn} A^{(n)} \subset \F(\hn \times A) \subset \F(A\times A)$ are definable sets in $\ca A^{eq}$ and $\dom(\iota)$ is a finite subset of $\bigcup_{n\in \hn} A^{(n)}$. 

To find $\zeta$, suppose first that $\dom(\iota)$ contains a single element $f\in A^{(n)}$. Since $f\in \F(A\times A)$, by the proof of Lemma \ref{lem:nonstandard-sets-of-pairs}, $f$ is coded by a quintuple $Q = (X,Y,\alpha, \beta, C)$ with $X,Y,C\in \F(A)$ and $\alpha, \beta\in A$  (the details of the coding are not important). We can now extend $\iota$ to some map $\eta\in \ca G_0$ that contains in its domain hyperfinite sequences $f_X\in A^{(n_X)}, f_Y\in A^{(n_Y)}, f_C\in A^{(n_C)}$ enumerating $X,Y,C$, where $n_X, n_Y, n_C\in \hn$ may be non-standard. The associated map $\hat \eta\in \ca G$ will then contain in its domain $X,Y,C\in \F(A)$ and all the elements of $\ext X, \ext Y, \ext C$. Note that $\ext X = (n) \subset \hn$ and $\ext f \subseteq \ext X\times \ext Y$.  

We now extend $\hat \eta$ to a partial elementary map $\zeta \in \ca G$ that contains in its domain also $\alpha$ and $\beta$. 
The map $\zeta$ induces a partial elementary map $\zeta^{eq}$ from $\ca A^{eq}$ to $\ca B^{eq}$ that sends $f\in \F(A\times A)$ to the element $\zeta^{eq}(f) \in \F(B\times B)$ coded by $\zeta(Q) = (\zeta(X),\zeta(Y),\zeta(\alpha), \zeta(\beta), \zeta(C))$. Since $\zeta^{eq}$ is a partial elementary map and $f\in A^{(n)}$, we have $\zeta^{eq}(f)\in B^{(n)}$. Moreover, if $(x,y)\in \ext f$, then $x\in \hn$ and $(\zeta(x),\zeta(y)) = (x, \zeta(y)) \in \ext {\zeta^{eq} (f)}$, which we abbreviate as  $\zeta(f(x)) = \zeta^{eq}(f)(x)$. 

Using the fact that $\zeta \supseteq \hat \eta \supseteq \hat \iota$, and $f(x)\in \dom(\hat \iota)$ for all $x<n$, we have $\zeta(f(x)) = \hat \iota(f(x))$, so $\zeta^{eq}(f) = \iota(f)$. 

This concludes the proof when $\dom(\iota)$ contains a single element, and the general case is similar, recalling that the domain of $\iota$ is finite. 
\end{proof}

\begin{cor}\label{cor:type ideal}
Let $\ca K = (K, \F(K))$ be a model of $T_{\text{rec}}$. Let $\hn = \hn(\ca K)$ and $n\in \hn$. If $f\in K^{(n)}$, then the type of $f$ over $\hn$ is determined by its definable ideal $\ideal I = \ideal I_f\subseteq \poly n$. 
\end{cor}
\begin{proof}
Suppose $\ideal I_f = \ideal I_g$. Then the map $\iota$ sending $f$ to $g$ belongs to the set $\ca G_0$ defined in Lemma \ref{lem:main} taking $(A,\F(A)) = (B, \F(B)) = (K, \F(K))$, and therefore it is a partial elementary map from $\ca K$ to itself by Corollary \ref{cor:iota elementary}. 
\end{proof}

\begin{thm}\label{thm:completeness}
The complete theory of $(\C, \Fin(\C))$ is axiomatized by $T^{\Fin}_\C = T_{\text{rec}} \cup T_{\N}$. 
\end{thm}

\begin{proof}
By Shoenfield's absoluteness theorem \cite[Theorem 98]{Jech}, a statement of complexity $\Sigma^1_2$ is provable in ZFC (Zermelo-Fraenkel set theory with choice) if and only if it is provable in ZFC+GCH (where GCH is the generalized continuum hypothesis). To prove the completeness of $T^{\Fin}_\C$ we may then freely assume that the generalized continuum hypothesis holds. This guarantees that all structures have saturated extensions of any sufficiently large cardinality. Now note that if $\mathcal A= (A, \F (A))$ and $\mathcal B = (B, \F(B))$ are saturated models of $T_{\text{rec}}$ of the same cardinality, then $(\hn (\A),+,\cdot)$ and $(\hn (\B),+,\cdot)$ are saturated models of $Th(\N, +, \cdot)$ of the same cardinality, so they are isomorphic as structures in the language of arithmetic. We conclude using Lemma \ref{lem:main}. 
\end{proof}

Shoenfield's absoluteness lemma could be avoided using the fact that elementarily equivalent structures have isomorphic ultrapowers \cite{Shelah}. The two approaches lead to the same conclusion. 
With the same proof we obtain: 

\begin{thm}\label{thm:completeness-mod-peano}
Given a model $\ca K = (K, \F(K))$ of $T_{\text{rec}}$, the complete theory $T_{\ca K}$ of $\ca K$ is determined by the complete theory of $\hn (\ca K)$ and it is given by
$$T_{\ca K} = T_{\text{rec}} \cup T_{\hn(\ca K)}$$ 
where $T_{\ca K} = \{\varphi^{\hn} \mid (\hn(\ca K),+,\cdot) \models \varphi\}$. 
\end{thm}

Next, we prove that the structure induced by $\ca K$ on $\hn(\ca K)$ is precisely $(\hn (\ca K), +, \cdot)$.

\begin{thm} \label{thm:stable-emebeddedness}
	Let $\ca K = (K, \F(K))$ be a model of $T_{\text{rec}}$ and let $\hn = \hn(\ca K)$. Then every 
definable subset of $\hn^n$ definable with parameters from $K \cup \F(K)$ is already definable with parameters from $\hn$ and it is in fact definable in the structure  $(\hn, +,\cdot)$. 
\end{thm}

\begin{proof} We first prove the statement assuming GCH and then we show how to relax this hypothesis. As above, the use of GCH could also be replaced by a Keisler--Shelah argument. 

Let $X\subseteq \hn^n$ be definable in $\ca K$ by a formula $\varphi(x, c)$, where $c$ is a tuple of parameters from $K\cup \F(K)$. Since every hyperfinite set can be enumerated by a hyperfinite sequence (Proposition \ref{prop:enumeration}), there are $m\in \hn$ and $f\in K^{(m)}$ such that $c$ is definable from $f$ (it is easy to see that a single $f$ suffices). We may then assume that $X$ is definable by a formula $\varphi(x, f)$ with parameter $f\in K^{(m)}$. Let $\ideal I_f \subseteq \poly m$ be the definable ideal of $f$. Then $\ideal I_f$ is definable in $(\hQ, \hn, +,\cdot)^{eq}$ (Corollary \ref{cor:ideal}). 
By Corollary \ref{cor:type ideal} the ideal $\ideal I = \ideal I_f$ determines the type of $f$ over $\hn$. So for all tuples $x$ from $\hn$, $\varphi (x,f)$ is equivalent to the formula $\exists g \; (\ideal I_g = \ideal I \land \varphi(x,g))$. This formula can be taken to have only parameters from $\hn$ (those needed to define $\ideal I$), since definable finite power sets are coded in $(\hQ,\hn,+,\cdot)^{eq}$ by elements of $\hn$ (see the proof of Lemmas \ref{lem:robust-pow-set-of-N} and \ref{lem:robust-hyperfinite-in-Q}). 

It remains to show that every subset $X \subseteq \hn^n$ definable in $\ca K$ by a formula $\varphi(x)$ with parameters from $\hn$, is already definable in the structure $(\hn, +, \cdot)$. To this aim let $\Gamma$ be the family of all formulas in the language $\{+,\cdot, c\}_{c \in \hn(\ca K)}$ where all quantifiers are relativized to the $0$-definable predicate $\hn$, all free variables are constrained in $\hn$ and all constants are from $\hn (\ca K)$. We need to prove that $X$ is $\Gamma$-definable. 
Consider two elementary extensions $\ca A, \ca B$ of $\ca K$,  and elements $a_1, \ldots, a_n \in \hn(\ca A)$, $b_1, \ldots, b_n\in \hn(\ca B)$, with  
$$\ca A, a_1, \ldots, a_n \equiv_{\Gamma} \ca B, b_1, \ldots, b_n.$$
By Proposition \ref{prop:EQ} it suffices to show that 
$$\ca A, a_1, \ldots, a_n \equiv_{\varphi(x)} \ca B, b_1, \ldots, b_n$$
By GCH we may assume that $\ca A$ and $\ca B$ are saturated of the same cardinality, hence there is an isomorphism
 $$\Psi:(\hn(\ca A),+,\cdot, a_1,\ldots, a_n) \cong (\hn(\ca B),+,\cdot, b_1, \ldots, b_n)$$
 and we may assume that $\Psi$ is the identity on the parameters from $\hn (\ca K)$ in the formula $\varphi(x)$. 
 By Lemma \ref{lem:main} $\Psi:\ca A \pto \ca B$ is an elementary map, so it preserves $\varphi(x)$, concluding the proof. 
 
 To show that that GCH is not needed, we first observe that the statement of the theorem can be rephrased as follows. For all formulas $\varphi (x, y)$ in the language of $T_{\text{rec}}$ (where $x,y$ are tuples of variables) and for all complete extensions $T_y$ of $T_{\text{rec}}$ in the language $L\cup \{y\}$, there is a formula $\gamma (x, z)$ of $\Gamma$ such that $T_y$ proves $\exists z \in \hn \; (\forall x \in \hn \; (\varphi(x,y) \iff \gamma(x,z))$. 
 Under this rephrasing the statement has complexity (at most) $\Sigma^1_2$, so  Shoenfield's absoluteness lemma applies. 
 \end{proof}

\section{The case of transcendence degree zero}\label{sec:Qbar}

The theory of $(\overline \Q, \Fin(\overline \Q))$ differs from the theory of $(\C, \Fin(\C))$ because $(\overline \Q, \Fin(\overline \Q))$ does not satisfy axiom (5) in Definition \ref{defn:complete}. However we can adapt the results of Section \ref{sec:wso-complex} to study the theory of $(\overline \Q, \Fin(\overline \Q))$. 

\begin{defn}
	Let $T_{\text{rec},0}$ be defined as $T_{\text{rec}}$ but with point (5) in Definition \ref{defn:complete} replaced by an axiom saying that every element of $K$ is definably algebraic over $\hQ$.
\end{defn}

\begin{prop} \label{prop:Qbar} \mbox{}
	\begin{enumerate}
		\item $T_{\text{rec},0}\cup T_\N$ is the complete theory of $(\overline \Q, \Fin(\overline \Q))$, 
		\item The complete theory of a model $\ca K$ of $T_{\text{rec},0}\cup T_\N$ is determined by the complete theory of $\hn(\ca K)$. 
		\item if $\ca K$ is a model of $T_{\text{rec},0}$ and $\hn = \hn(\ca K)$, every definable subset of $\hn^n$ with parameters from $\ca K$ is already definable in $(\hn, +, \cdot)$. 
\end{enumerate}	 
\end{prop}
\begin{proof}
We need a variant of Proposition \ref{prop:generic} stating that if $\ideal I$ is the definable ideal $\ideal I_f$ of some $f\in K^{(n)}$, then every zero $g\in K^{(n)}$ of $\ideal I$ is generic, namely we have $\ideal I_g = \ideal I$.  The argument is similar to Case 1 of the proof of Proposition \ref{prop:generic}. Granted this, the proof of Lemma \ref{lem:main} goes through replacing the application of \ref{prop:generic} by this variant. The analogues of Theorems \ref{thm:completeness}, \ref{thm:completeness-mod-peano}, \ref{thm:stable-emebeddedness} for $T_{\text{rec},0}$ follow by the same proofs. 
\end{proof}	

We expect that in a similar way we can analyze the theory of $(K, \Fin(K))$ where $K$ is an algebraically closed extension of $\Q$ of finite transcendence degree. 

\section{A poset of algebraic curves and points}\label{sec:poset}

In this section we prove that the incidence relation between irreducible complex projective curves and their points interprets $(\C, \Fin(\C))$. 

\begin{defn}
	Let $K$ be a field and let  $\Var K$ be the poset of irreducible Zariski closed proper subsets of $\mathbb{P}^2(K)$ ordered by inclusion. We can write 
	$$\Var K = \Points K \cup \Curves K$$ where $\Points K$ is the set of all (irreducible) varieties of dimension $0$, which we can identify with the points of $\mathbb P^2(K)$, and $\Curves K$ are the varieties of dimension $1$. Since we only consider irreducible curves, there are no inclusions between distinct curves, so the order of the poset is the membership relation between points and curves. 

 We should distinguish a curve $C\in \Curves K$ from its set of points 
 $$\pt C \subseteq \Points K,$$
  but when there is no risk of confusion, we will omit the superscript and identify \( C \) with \( \pt C \), allowing the context to clarify whether we mean a subset of $\Points K$  or an element of $\Curves K$. 
 
 Given $C,D\in \Curves K$, let  $C\cap D\subseteq \Points K$ be the set of points in both $C$ and $D$. 
	
	  We write $\nCurves K n\subseteq \Curves K$  for the set of curves of degree $n$ and 
$$\nVar K n = \Points K \cup \nCurves K n.$$
Similarly we define $\nVar K {\leq n} = \Points K \cup \nCurves K {\leq n}$ by considering the curves of degree $\leq n$. 
\end{defn}

We also need to consider the poset of affine irreducible varieties obtained by removing a line at infinity and its points, as in the following definition. 

\begin{defn}\label{defn:affine}  Fix a line $L_\infty\in \nCurves K 1$ to play the role of the line at infinity and define: 
\begin{itemize}
\item $\affPoints K = \Points K \setminus L_\infty^{\text{points}}$.
\item $\affCurves K = \Curves K \setminus \{L_\infty\}$.
\item $\naffCurves K n = \nCurves K n \setminus \{L_\infty\}$.
\item $\affVar K = \affPoints K \cup \affCurves K \; \subset\; \Var K$. 
\item $\naffVar K n = \affPoints K \cup \naffCurves K n \; \subset\; \nVar K n$. 
\item $\naffVar K {\leq n} = \affPoints K \cup \naffCurves K {\leq n} \; \subset\; \nVar K {\leq n}$. 
\end{itemize}
Note that $\naffVar K 1 = \naffVar K {\leq 1}$. 
\end{defn}

\begin{rem} \mbox{} \label{rem:affine}
With these definitions $\affVar K$ is a substructure of the poset $\Var K$ isomorphic to the poset of all affine proper irreducible varieties $C\subset K^2$ (points and curves) ordered by inclusion, where we identify a point with its singleton set.
A possible source of confusion is that an affine curve is also a projective curve, 
namely $\affCurves K \subset \Curves K$, however the set of points $\pt C$ of a curve $C$ is  different when computed in the poset $\affVar K$ or in the larger poset $\Var K$. In the former case the points at infinity are missing, although they are determined uniquely by the affine points of the given curve. 

Let us also observe that $\naffVar K {1}$ is isomorphic to 
		the poset, definable in $K^{eq}$, consisting of all affine lines $L\subset K^2$ and points $P\in K^2$ ordered by inclusion. 
\end{rem}

\section{A bi-interpretability result}\label{sec:biinter}

The bi-interpretability of the field $K$ with the poset $\naffVar K 1$ essentially boils down to recovering the field structure on a line. This is well-known, see for example \cite[Theorem 7.12]{Har}; we give a proof to fix some notation and to describe the ideas which we will use later.
 
\begin{prop}\label{prop:field-interpretation}
	For every field $K$, the poset $\naffVar K 1$ of all affine lines and points is bi-interpretable with the field $K$ (after fixing some parameters). More precisely, given any affine line $L\in \naffCurves K 1$ and two points ${\bf 0}, {\bf 1}\in L^{\text{points}}$, we have: 
	
	\begin{enumerate}
	\item There are a field structure on $L^{\text{points}}$, with $\bf 0, \bf 1$ as the additive and multiplicative identities, and a field isomorphism $K\cong L^{\text{points}}$. 
	\item 
	The isomorphism $K\cong L^{\text{points}}$ is definable in $K^{eq}$ if we view $\naffVar K 1$ as a definable structure in $K^{eq}$ as in Remark \ref{rem:affine}. 
	\item 
	We can introduce coordinates in $\naffVar K 1$ in the following sense: there is a bijection $f: \affPoints K \to \pt L \times \pt L$ which induces an isomorphism $\naffVar K 1 \cong \naffVar {\pt L} 1$ definable in ${\naffVar K 1}^{eq}$. 
\end{enumerate}
\end{prop}

\begin{proof} 
	The interpretation of $\naffVar K 1$ in $K$ follows from Remark \ref{rem:affine}. 
	
	To interpret $K$ in $\naffVar K 1$ we proceed as follows.
	Given two distinct points $a,b\in \affPoints K$ let $L(a,b)$ be the unique line joining $a$ and $b$. Two distinct lines are parallel if they do not intersect. If $L(a,b)$ is parallel to $L(c,d)$ and $L(a,c)$ is parallel to $L(b,d)$ we write $[a,b]\sim [c,d]$ (intuitively, $[a,b]$ and $[c,d]$ are opposite sides of a parallelogram). We say that $[a,b]$ is congruent to $[e,f]$ if either $a=b$ and $e=f$ or there are $c,d$ with $[a,b]\sim [c,d] \sim [e,f]$ (intuitively, $[e,f]$ is a translation of $[a,b]$). 
	
	Given $x,y\in L$ we define $x+y$ as the unique point of $L$ such that $[\mathbf 0,x]$ is congruent to $[y,x+y]$. Now fix a line $L' \neq L$  which intersects $L$ at $\mathbf 0$.
	To define the product $xy$, choose $a,b\in L'$ so that $L(a,\mathbf 1)$ and $L(b,x)$ are parallel. Now define $xy\in L$ as the unique point such that $L(a,y)$ and $L(b,xy)$ are parallel.  
	
	It is well-known that with the above definition we obtain a field with domain $L$ (or more precisely $L^{\text{points}}$) isomorphic to $K$. This proves point (1). 
	
	So far we have shown that $K$ and $\naffVar K 1$ are mutually interpretable, but it remains to show that they are bi-interpretable, namely that the induced self-interpretations are definable. The definability of the self-interpretation of $K$ in itself is given by point (2) and is easy. The definability of the self-interpretation of $\naffVar K 1$ in itself follows from point (3) which is proved as follows. The idea is to use $L$ as the $x$-axis and $L'$ as the $y$-axis after fixing a definable bijection $\phi:L'\to L$ using a family of parallel lines. Given $p\in \affPoints K$ let $(x,z) \in L\times L'$ be such that $[0,x]\sim [z,p]$ and assign to $p$ the coordinates $(x, \phi(z)) \in L\times L$. This defines a bijection $f:\affPoints K \to \pt L \times \pt L$ and it is easy to check that $f$ induces the required isomorphism.
\end{proof}

\begin{thm}\label{thm:curves-field} Let $K$ be a field and let $n$ be a positive integer. Then the poset $\nVar K {\leq n}$ is bi-interpretable with $K$ (after fixing some parameters, see Remark \ref{rem:four points}). 
\end{thm}

\begin{proof}
The interpretation of $K$ in  $\nVar K {\leq n}$ is obtained through the following steps. 
\begin{itemize}
\item[(i)] $\nVar K {\leq 1}$  is definable in $\nVar{K}{\leq n}$ because an irreducible curve has degree at most $1$ if and only if it intersects every irreducible curve of degree at most $n$ in at most $n$ points. 
\item[(ii)] $\naffVar K 1$ can be identified with the definable substructure of $\nVar K {\leq 1}$ obtained by removing a projective line $S_\infty$ and its points. 
\item[(iii)] $\naffVar K 1$ defines a field $L$ isomorphic to $K$  whose domain is the set of points of any given affine line (Proposition \ref{prop:field-interpretation}(1)).
\end{itemize}

To interpret the poset $\nVar K {\leq n}$ in $K$ we proceed as follows. Let $m$ be the cardinality of the set of monomials in three variables of degree $\leq n$. We can naturally identify $K^m$ with the set $K[\tt{x_0, x_1, x_2}]^{\leq n}$ of polynomials of degree $\leq n$. 
In $K^{eq}$ we can define an evaluation function $\ev_n: K[\tt{x_0, x_1, x_2}]^{\leq n}$ $\times K^{3}\to K$ sending $(p, (a,b,c))$ to $p(a,b,c)$. The subset consisting of the irreducible homogeneous polynomials is clearly definable. Each irreducible  homogeneous polynomial $p\in K[{\tt x_0, x_1, x_2}]^{\leq n}$ defines an irreducible projective algebraic set $V(p) = \{[a,b,c]\in \mathbb P^2(K)  \mid f(a,b,c) = 0\}$ and two homogeneous polynomials define the same set if they have the same zeros. This gives an interpretation of $\nVar K {\leq n}$ in $K^{eq}$.

We have thus shown that $K$ and $\nVar K {\leq n}$ are mutually interpretable. The proof that they are bi-interpretable is similar to that of Proposition \ref{prop:field-interpretation}. The main difference is that we have to use homogeneous rather than affine coordinates to deal with the self-interpretation of $\nVar K {\leq n}$ in itself. Recall that 
$\nVar K {\leq n} = \Points K \cup \nCurves K {\leq n}$ and $\naffVar K 1$ is definable in $\nVar K {\leq n}$.  Now let $L$ be the field introduced in point (iii). By Proposition \ref{prop:field-interpretation}(3) we have a definable bijection $\phi$ from $L\times L$ to $\affPoints K = \Points K \setminus S_\infty^{\text{points}}$. We can naturally embed $L\times L$ in $\mathbb P^2 (L)$ and extend $\phi$ to a definable bijection from $\Points K$ to $\mathbb P^2(L)$ which preserves collinearity. 
This induces an isomorphism from $\nVar K {\leq n}$ to $\nVar L {\leq n}$ which is definable in $\nVar K {\leq n}^{eq}$. 
\end{proof}

The following proposition shows that $\nVar K {\leq 2}$ is definable in $\Var K$. 

\begin{prop}{\cite[Proposition 2.6]{MD}} \label{prop:2curves}
Let $K$ be an algebraically closed field of characteristic zero. Let $C \in \Curves{K}$ be an irreducible plane algebraic curve. Then $C$ has degree at most $2$ if and only if for every point $P\in C$ 
there is $D\in \Var{K}$ such that $C \cap D = \{P\}$.
\end{prop}
\begin{proof} See the appendix. \end{proof}

\begin{thm} \label{thm:biinterprable}
Let $K$ be an algebraically closed field of characteristic $0$. Then $(K,\Fin(K))$ and $\Var{K}$ are bi-interpretable (with parameters, see Remark \ref{rem:four points}). 
\end{thm}
\begin{proof}
We show: 
\begin{itemize}
\item[(i)] $(K, \Fin(K))$ is interpretable in $\Var K$. 
\item[(ii)] $\Var K$ is interpretable in $(K, \Fin(K))$. 
\item[(iii)] Composing the interpretations we obtain a bi-interpretation between $(K, \Fin(K))$ and $\Var K$. 
\end{itemize}
 By Proposition \ref{prop:2curves} $\nVar K {\leq 2}$ is definable in $\Var K$.
 By Theorem \ref{thm:curves-field}, the field $K$ is bi-interpretable with $\nVar K {\leq 2}$, so in particular it is interpretable in $\Var K$. 
 We show that the above interpretation can be extended to an interpretation of $(K, \Fin(K))$ in $\Var K$. 
 Let $L\in \naffVar K 1$ be an affine line with a field structure on $L^{\text{points}}$ definable in $\Var K$ and isomorphic to $K$.  Any finite subset $X$ of $\pt L$ can be written in the form $C \cap L$ for some $C\in \Var K$ not containing $L$. 
In this situation we say that $C$ codes $X$. By definition $C$ and $C'$ code the same set if $C\cap L = C'\cap L$, so the equivalence of codes is definable in $\Var K$ and we have obtained the desired interpretation of $(K, \Fin(K))$. This proves (i). 

\smallskip 
To prove (ii) we need the fact that in $(K, \Fin(K))^{eq}$ we can define an evaluation function $\ev: K[\tt{x_0, x_1, x_2}]^{\text{def}}$ $\times K^{(3)}\to K$ for polynomials of arbitrary degree (Definition \ref{defn:ev}). Recall that by Remark \ref{rem:standard poly} we can identify $K[\tt{x_0, x_1, x_2}]^{\text{def}}$ with the ring of all standard polynomials in three variables. The subset consisting of the irreducible homogeneous polynomials is clearly definable. Each irreducible homogeneous polynomial $p\in K[{\tt x_0, x_1, x_2}]$ defines an irreducible projective algebraic set $V(p) = \{[a,b,c]\in \mathbb P^2(K)  \mid p(a,b,c) = 0\}$ and two polynomials define the same set if they have the same zeros. This gives an interpretation of the poset of projective irreducible algebraic sets in $(K, \Fin(K))^{eq}$. 

Let then $\textnormal{Poset}(\mathcal{K})$ denote the union of $\mathbb{P}^2(K)$ and the set of the equivalence classes of homogeneous polynomials which define irreducible algebraic sets, endowed with the poset structure induced by the membership of a point to a curve. This is isomorphic to $\Var{K}$ and gives the desired interpretation. 
\smallskip 

The proof of point (iii) is similar to the proof of Theorem \ref{thm:curves-field}. As in that theorem we obtain a bijection $\phi$ from $\Points K$ to $\mathbb{P}^2(L)$ definable in $\Var K$ which preserves collinearity. Since a curve is determined by its points, $\phi$ extends to a definable isomorphism between $\Var K$ and $\Var L$.
\end{proof}	

Inspecting the proof of Theorem \ref{thm:biinterprable} we have: 

\begin{cor}\label{cor:recover-equation}
	There is a function $F$, definable in $\Var{K}^{eq}$, such that for all $C \in \Curves{K}$ we have that $F(C)\in K[\tt x_0, \tt x_1, \tt x_2]$ is a homogeneous polynomial defining $C$.
\end{cor}	
\begin{proof}
Given the poset $\Var K$ we interpret $(K, \Fin(K))$, which in turn interprets a poset $P$ definably isomorphic to $\Var K$. A curve $C\in \Var K$ is interpreted in $P$ as an equivalence class of irreducible homogeneous polynomials defining $C$ and the interpretation is definable. We can definably pick a polynomial of minimal degree in the class with a $1$ as the first nonzero coefficient with respect to a definable ordering of the monomials.  
\end{proof}

\begin{rem}\label{rem:four points}
	In Theorem \ref{thm:curves-field} the interpretation of $K$ in $\Var K$ requires four parameters in $\Points K$. 
	We use the four points to define the system of projective coordinates determined by a line at infinity and elements $0,1$ on another line which will play the role of the neutral elements of the field operations.  The same holds for the interpretation of $(K, \Fin(K))$ in $\Var K$. On the other hand the interpretation of $\Var K$ in $(K, \Fin(K))$ is without parameters. 
\end{rem}

\begin{cor}\label{cor:var}
	Let $K$ be an algebraically closed field of characteristic zero. Then 
	 $\Var K \equiv \Var \C$ if and only if $K$ has infinite transcendence degree. 
\end{cor}
\begin{proof}
	Suppose that $K$ has infinite transcendence degree. Then by Theorem \ref{thm:completeness} $(K, \Fin(K)) \equiv (\C, \Fin(\C))$. Therefore $(K, \Fin(K))^{eq} \equiv (\C, \Fin(\C))^{eq}$. Since the structures $\Var K$ is definable without parameters in $(K, \Fin(K))^{eq}$ and $\Var \C$ is definable in $(\C, \Fin(\C))^{eq}$ by the same formula, we have $\Var K \equiv \Var \C$. 
	
	Suppose now that $\Var K \equiv \Var \C$. In the projective plane of a field $F$, any two quadruples of points in general position (i.e. no three of which lie on the same line), are conjugated by a unique projective collineation (\cite[Theorem 8.12]{Har}). This in turn induces (via the associated system of coordinates) a unique automorphism of $\Var F$. So if we fix such a quadruple $\alpha$ in $\Points K$ and a similar quadruple $\beta$ in $\Points \C$, we have that the types of $\alpha$ and $\beta$ over the empty set coincide, namely $(\Var K, \alpha) \equiv (\Var \C, \beta)$. By Theorem \ref{thm:biinterprable} and Remark \ref{rem:four points}, $(K, \Fin (K)) \equiv (\C, \Fin(\C))$, so by Theorem \ref{thm:completeness} $K$ has infinite transcendence degree.  
\end{proof}

\section{Appendix}\label{sec:appendix}

In this Appendix we give a self-contained proof of Proposition \ref{prop:2curves}. We follow an approach which was suggested by Rita Pardini in the case of smooth curves, and then we develop it further to account for singular curves. We also recall some basic definitions and facts from the theory of divisors on curves and of generalized Jacobians; our main references are \cite{HS} and \cite{Serre}. Throughout this section, all curves will be assumed to be irreducible and projective, even when not explicitly stated.

For the rest of the Appendix we fix an algebraically closed field $K$ of characteristic $0$. We stress that Proposition \ref{prop:2curves} does not hold in positive characteristic: see \cite[Corollario 1.6]{Mar} for a counterexample. 

Given a projective curve $C$ over $K$ and a point $P$ in $C$, a regular function on $C$ at $P$ is one that is given, on an affine neighbourhood $U \subseteq C$ of $P$, by a quotient of homogeneous polynomials over $K$ of the same degree with non-vanishing denominator at $P$. The regular functions at a point $P$ form a ring $\mathcal{O}_P$, with unique maximal ideal given by the functions vanishing at $P$. The residue field of this local ring is then isomorphic to $K$. The point $P$ is smooth on $C$ if and only if $\mathcal{O}_P$ is integrally closed (see \cite[Exercise A.1.1.5]{HS}), and in this case $\mathcal{O}_P$ is a discrete valuation ring.

The set of pairs $(U,f)$, where $U$ is a Zariski open subset of $C$ and $f$ is a function which is regular on $U$, modulo the equivalence relation $(U,f) \sim (V,g)$ if $f=g$ on $U \cap V$, forms a field with the operations of pointwise sum and product; the elements of this field are called \emph{rational functions} on the curve $C$. This field is denoted by $K(C)$.

Every smooth point $P$ of a curve $C$ determines a valuation on $K(C)$, denoted by $\ord_P$, as follows. Since $\mathcal{O}_P$ is local and integrally closed, it is a discrete valuation ring, and hence a principal ideal domain, so the maximal ideal has a generator $g$. The valuation $\ord_P(f)$ of a function $f \in \mathcal{O}_P$ is given by the highest power of $g$ which divides $f$ in $\mathcal{O}_P$, and the valuation of a function  $f \notin \mathcal{O}_P$ is $-\ord_P(1/f)$. The maximal ideal of $\mathcal{O}_P$ is then generated by any element of valuation $1$. For a fixed rational function $f$, it is easy to see that there are only finitely many points $P \in C$ for which $\ord_P(f) \neq 0$ and that $\ord_P(f) > 0$ if and only if $f(P) = 0$. 

\begin{defn}
	Let $C$ be a smooth projective curve. The group $\Div(C)$ of \emph{Weil divisors} on $C$ is the free abelian group generated by the points of $C$.
	
	Given a Weil divisor $D=\sum_{P \in C} n_P P$ (with $n_P \neq 0$ for finitely many $P \in C$), the \emph{degree} of $D$ is the integer $\sum_{P \in C} n_P$ and the \emph{support} of $D$ is the finite set of all the points $P\in C$ with $n_P\neq 0$. 
\end{defn}

It is immediate that the Weil divisors of degree $0$ form a subgroup of the group of Weil divisors, which is generated by the divisors of the form $P-Q$ for $P, Q \in C$. This subgroup is denoted by $\Div^0(C)$.

\begin{defn}
	Let $C$ be a smooth projective curve, $f$ a rational function on $C$. The \emph{divisor of $f$} is the Weil divisor $(f)=\sum_{P \in C} \ord_P(f) P$. Divisors of this form are called \emph{principal} .
	
	Two Weil divisors $D,D'$ are \emph{linearly equivalent} if $D-D'$ is a principal divisor.
\end{defn}

The principal divisors form a subgroup of $\Div^0(C)$, since the identity of $\Div(C)$ is the divisor of a non-zero constant and we have $(f^{-1})=-(f)$ and $(fg)=(f)+(g)$.

\begin{prop}\label{fact:jacobians}
	Let $C$ be a smooth projective curve.
	\begin{enumerate}
		\item The quotient of $\Div^0(C)$ by the subgroup of principal divisors is an algebraic group, the \emph{Jacobian} $\Jac(C)$ of $C$. 
		\item The group of $K$-points of $\Jac(C)$ contains a non-torsion point if and only if $C$ has positive genus.
	\end{enumerate}
\end{prop}

\begin{proof}
	By \cite[Proposition A.2.1.3]{HS} every divisor on $\mathbb{P}^1$ of degree $0$ is principal, and by \cite[Theorem A.4.3.1]{HS} every smooth curve of genus $0$ is isomorphic to $\mathbb{P}^1$. Hence if $C$ has genus $0$ then $\Jac(C)$ is the trivial algebraic group.
	
	It remains to establish that if $C$ has positive genus then $\Jac(C)$ is an algebraic group and it has non-torsion points which are $K$-rational. The former statement is \cite[Theorem A.8.1.1]{HS}. The latter is \cite[Theorem A in Appendix]{Rosen}.
\end{proof}

Note that if $K=\mathbb{C}$, or more generally if $K$ is uncountable, the fact that the Jacobian of a smooth curve $C$ of positive genus has non-torsion points follows from cardinality considerations, since the Jacobian has finite $n$-torsion for every $n$, and hence its torsion group is countable, see \cite[Theorem A.7.2.7]{HS}.

\begin{defn}
	Let $C \subseteq \mathbb{P}^2(K)$ be an irreducible projective plane curve. We say $C$ is \emph{prefactorial} if for every $P \in C$ there is a curve $D \subseteq \mathbb{P}^2(K)$ such that $C \cap D=\{P\}$.
\end{defn}

Our goal is to show that plane curves are prefactorial if and only if they have degree at most 2. We first present the argument in the smooth case. Recall that smooth plane curves satisfy the \emph{genus-degree formula}: the genus of the smooth curve $C$ is given by $\frac{(\deg(C)-1)(\deg(C)-2)}{2}$.

\begin{prop}\label{prop:if-smooth-pref-then-deg}
	Let $C$ be a smooth plane projective curve. Assume $C$ is prefactorial. Then $C$ has genus $0$, hence degree at most 2.
\end{prop}

\begin{proof}
	We will show that the Jacobian $\Jac(C)$ is a torsion group, which, since $K$ is algebraically closed of characteristic $0$, is only possible if $\Jac(C)$ is the trivial group and $C$ has genus $0$. 
	
	Let $P_1 \neq P_2$ be points on $C$. Since $C$ is prefactorial, there are irreducible homogeneous polynomials $F_1$ of degree $d_1$ and $F_2$ of degree $d_2$ such that $F_i$ has a unique zero on $C$, at $P_i$, whose order is equal to the intersection multiplicity of the curve defined by $F_i$ and $C$; by B\'ezout's Theorem \cite[Theorem A.4.6.1]{HS} this is $d_i \deg(C)$.
	
	On the Zariski open dense subset of $C$ defined by $F_2(P) \neq 0$, the function $P \mapsto F_1(P)^{d_2}/F_2(P)^{d_1}$ is given by a quotient of homogeneous polynomials of the same degree $d_1d_2$, hence it is a rational function on $C$. Its divisor  is $\deg(C)d_1d_2(P_1-P_2)$, so $P_1-P_2$ is torsion in $\Jac(C)$. Since divisors of the form $P_1-P_2$ generate the subgroup $\Div^0(C)$, this must be all torsion. Therefore $C$ has genus $0$ by Proposition \ref{fact:jacobians}(2); since it is smooth, by the genus-degree formula its degree is at most 2.
\end{proof}

 We will now use \emph{generalized Jacobians} to show that in fact every prefactorial plane curve is smooth. Generalized Jacobians are algebraic groups which arise considering a smooth curve $C$ together with additional data, which encodes information on the singularities on a curve $C'$ with a birational surjective map $C \to C'$. When this data is trivial, i.e$.$ when $C'=C$ is smooth, we recover the usual Jacobian, but we will see that as soon as it is not trivial the generalized Jacobian has infinite rank. We will then adapt the proof of Proposition \ref{prop:if-smooth-pref-then-deg} to show that prefactoriality implies finite rank of the generalized Jacobian, concluding that prefactorial curves must be smooth of genus 0. 

For a singular point $P$ on a curve $C$, we will write $\widetilde{\mathcal{O}_P}$ for the integral closure of $\mathcal{O}_P$. A \emph{normalization} of a singular curve $C$ is a smooth curve $\tilde{C}$ with a finite birational morphism $\phi: \tilde{C} \rightarrow C$; this always exists and is unique up to isomorphism. The map $K(C) \rightarrow K(\tilde{C})$ defined by $f \mapsto f \circ \phi$ is then an isomorphism of fields (this is a property of birational maps), and under this isomorphism we have $\widetilde{\mathcal{O}_P} \cong \bigcap_{Q \in \phi^{-1}(P)} \mathcal{O}_Q$ for all $P \in C$. See for example \cite[Ch.~IV.1]{Serre}, and the references therein, for these and other facts about normalizations.

\begin{defn}
	Let $C$ be a smooth projective curve, $S \subseteq C$ a finite subset. 
	\begin{enumerate}
		\item A \emph{modulus} on $C$ supported on $S$ is a Weil divisor $\sum n_P P$ with support $S$ such that $n_P >0$ for all $P \in S$.
		\item If $\mathfrak{m} = \sum n_P P$ is a modulus supported on $S$, $f$ is a rational function on $C$, and $c \in K$ is a constant, we say that $f$ is \emph{congruent to $c$ modulo $\mathfrak{m}$}, and write $f \equiv c \mod \mathfrak{m}$, if $\ord_P(f-c) \geq n_P$ for all $P \in S$.
		\item A Weil divisor $D=\sum n_P P$ on $C$ is \emph{prime to $S$} if $n_P=0$ for all $P \in S$. 
		\item Given a modulus $\mathfrak m$ on $C$ supported on $S$, we say that two divisors $D,D'$ prime to $S$ are \emph{$\mathfrak{m}$-equivalent}, written $D \sim_{\mathfrak{m}} D'$, if $D-D'$ is the divisor of a rational function $g$ which is congruent to $1$ mod $\mathfrak{m}$.
	\end{enumerate} 
\end{defn}

We write $\Div(C \setminus S)$ for the set of divisors on $C$ prime to $S$, and $\Div^0(C \setminus S)$ for the subgroup of divisors of degree $0$. As for $\Div^0(C)$, the subgroup $\Div^0(C \setminus S)$ is generated by divisors of the form $P-Q$ for $P,Q \in C \setminus S$. 

\begin{lem}\label{lem:existence-modulus}
	Let $C$ be a projective curve, $P \in C$ a singular point, $\phi:\tilde{C} \rightarrow C$ a normalization, $\{Q_1,\dots,Q_\ell\}=\phi^{-1}(P)$. Then there is a modulus $\mathfrak m =   \sum_{i=1}^{\ell} n_i Q_i$ of degree  $n_1+\dots+n_\ell \geq 2$ supported on $\phi^{-1}(P)$ such that for all $f \in \mathcal{O}_P$, letting $c = f(P)$, we have $f \circ \phi  \equiv c \mod  \mathfrak m$.
\end{lem} 

\begin{proof} It is enough to prove the Lemma in the case $c = f(P) = 0$, otherwise we can replace $f$ with $f-c$.
	For all $f \in \mathcal{O}_P$ with $f(P)=0$ we must have $\ord_{Q_i}(f \circ \phi) > 0$ for all $i=1, \ldots, \ell$, so $f \circ \phi \equiv 0 \mod \sum_{i=1}^\ell Q_i$. This is enough to prove the statement as soon as $\ell >1$, so assume $\ell=1$ and let $\phi^{-1}(P)=\{Q\}$. Under this assumption, $\widetilde{\mathcal{O}_P} \cong \mathcal{O}_Q$, so $\widetilde{\mathcal O_P}$ is a discrete valuation ring and hence a principal ideal domain. 
	
	We claim that for every $f \in \mathcal{O}_P$ with $f(P) = 0$, we have $\ord_Q(f \circ \phi) \geq 2$. So the modulus $2Q$ satisfies the statement.
	Suppose for a contradiction that there is $f \in \mathcal{O}_P$ with $f(P) = 0$ such that $\ord_Q(f \circ \phi)=1$. Then $f$ generates the maximal ideal $\widetilde M$ of $\widetilde{\mathcal O_P}$ and belongs to the maximal ideal $M$ of $\mathcal{O}_P$, so $\widetilde{M}=M\widetilde{\mathcal{O}_P}$. Since $1$ generates the $K$-vector space $\widetilde{\mathcal{O}_P}/M\widetilde{\mathcal{O}_P} = \widetilde{\mathcal{O}_P}/\widetilde{M} \cong K$, by Nakayama's lemma \cite[Corollary 4.8]{Eis} we have that $1$ generates the $\mathcal{O}_P$-module $\widetilde{\mathcal{O}_P}$, so $\widetilde{\mathcal{O}_P}=\mathcal{O}_P$. This contradicts the singularity of $P$.		
\end{proof}

\begin{fact}[{\cite[Ch.~V.9 Theorem 1, p.88]{Serre}}]
	Let $C$ be a smooth projective curve, $S \subseteq C$ a finite set, $\mathfrak{m}$ a modulus supported on $S$. The quotient of $\Div^{0}(C \setminus S)$ by the subgroup of divisors which are $\mathfrak{m}$-equivalent to 0 is an algebraic group, denoted $J_{\mathfrak{m}}(C)$ and called the generalized Jacobian of $C$ with respect to $\mathfrak{m}$.
\end{fact}

We recall that the rank of an abelian group $G$ is the cardinality of a maximal subset of elements linearly independent over $\Z$, or equivalently the dimension of $G\otimes_\Z \Q$ as a $\Q$-vector space. 

\begin{lem}\label{fac:infinite rank gen jac}
	Let $C$ be a smooth projective curve, $\mathfrak{m}=\sum_{P \in S} n_P P$ a modulus on $C$ supported on $S$ of degree at least $2$. Then $J_{\mathfrak{m}}(C)$ has infinite rank.
\end{lem}

\begin{proof}
	By \cite[Ch.~V.13, Proposition 7, p. 92]{Serre}, there is an exact sequence of groups \[0 \rightarrow H_{\mathfrak{m}} \rightarrow J_{\mathfrak{m}}(C) \rightarrow \Jac(C) \rightarrow 0 \] where $H_{\mathfrak{m}}$ has the form $\left(\prod_{P \in S} U_P\right)/\mathbb{G}_m$ for appropriate groups $U_P$ depending on $P$. By \cite[Ch.~V.15, Corollary on p.94]{Serre}, since we are working in characteristic $0$, each group $U_P$ is isomorphic to $\mathbb{G}_m \times \mathbb{G}_a^{n_P-1}$. 
	
	If $|S| \geq 2$ then $H_{\mathfrak{m}}$ contains a copy of $\mathbb{G}_m$ and thus has infinite rank, and therefore so does $J_{\mathfrak{m}}(C)$. If $|S|=1$, then $\mathfrak{m}=n_P P$ for some $n_P \geq 2$, so $H_{\mathfrak{m}} \cong \mathbb{G}_a^{n_P-1}$ has again infinite rank.
\end{proof}

\begin{prop}\label{prop:if-pref-then-smooth}
	Let $C$ be a projective plane prefactorial curve. Then $C$ is smooth.
\end{prop} 

\begin{proof}
	We prove this by contradiction, so assume $P_0$ is a singular point of $C$, let $\phi:\tilde{C} \rightarrow C$ be a normalization, and let $\mathfrak{m}$ be the modulus of degree at least 2 given by Lemma \ref{lem:existence-modulus}. Let $S:=\phi^{-1}(P_0)$ be the support of $\mathfrak m$, and denote by $S'$ the finite set of all $Q \in \tilde{C}$ such that $\phi(Q)$ is singular, so $S \subseteq S'$. We will show that prefactoriality of $C$ implies that the generalized Jacobian $J_{\mathfrak{m}}(\tilde{C})$ has finite rank, contradicting Lemma \ref{fac:infinite rank gen jac}.
	
	Let $Q_1 \neq Q_2 \in \tilde{C} \setminus S'$. Since $C$ is prefactorial, there are homogeneous polynomials $F_1$ and $F_2$ in $3$ variables, of degree $d_1$ and $d_2$ respectively, such that $\phi(Q_1)$ and $\phi(Q_2)$ are the unique zeros of $F_1$ and $F_2$ respectively on $C$. Since $\phi(Q_1) \neq P_0 \neq \phi(Q_2)$, we have $F_1(P_0) \neq 0 \neq F_2(P_0)$, so after rescaling we may assume $F_1^{d_2}(P_0)=F_2^{d_1}(P_0)$.
	
	Consider the rational function $f:=\left(\frac{F_1^{d_2}}{F_2^{d_1}} \circ \phi\right) $ on $\tilde{C}$. Since $F_1^{d_2}/F_2^{d_1} \in \mathcal{O}_{P_0}$, and $F_1^{d_2}(P_0)/F_2^{d_1}(P_0)=1$, we have, by Lemma \ref{lem:existence-modulus}, that $f \equiv 1 \mod \mathfrak{m}$. Moreover, the divisor of $f$ is $\deg(C)d_1d_2 (Q_1-Q_2)$, so $Q_1-Q_2$ is torsion in $J_\mathfrak{m}(\tilde{C})$. Since the elements of the form $Q_1-Q_2$ generate the subgroup $\Div^0(\tilde{C} \setminus S')$ of $\Div^0(\tilde{C} \setminus S)$, from this it follows that the image of $\Div^0(\tilde{C} \setminus S')$ is contained in the torsion subgroup of $J_{\mathfrak{m}}(\tilde{C})$.
	
	Now enumerate the set $S' \setminus S$ as $R_1,\dots,R_n$ and fix a point $P_1 \in \tilde C \setminus S'$. For notational simplicity, for the rest of this proof we identify a divisor in $\Div^0(\tilde{C} \setminus S)$ with its class in $J_{\mathfrak{m}}(\tilde{C})$. We claim that $J_{\mathfrak{m}}(\tilde{C})$ is generated by the set \[ G:=\{P-Q \mid P,Q \in \tilde{C} \setminus S'\} \cup \{R_i-R_j \mid i,j=1,\dots,n \} \cup \{P_1-R_i \mid i=1,\dots,n\}. \] This set is the union of a finite set with a subset of the torsion subgroup of $J_{\mathfrak{m}}(\tilde{C})$, so the claim directly implies that $J_{\mathfrak{m}}(\tilde{C})$ has finite rank.

    To prove the claim, observe that $J_{\mathfrak{m}}(\tilde C)$ is generated by $\{P-Q \mid P,Q \in \tilde{C} \setminus S \}$, so it suffices to show that every class of the form $P-Q$ lies in the subgroup generated by $G$. Now if $P,Q \in \tilde{C} \setminus S'$ then $P-Q \in G$. If $P,Q \in S' \setminus S$ then there are $i,j \in \{1,\dots,n\}$ such that $P=R_i$ and $Q=R_j$, so $P-Q \in G$. Finally, if $P \in \tilde{C} \setminus S'$ and $Q \in S' \setminus S$ then $Q=R_i$ for some $i \in \{1,\dots,n\}$, and we may write $P-R_i=(P-P_1) + (P_1 -R_i)$, so $P-R_i$ lies in the subgroup generated by $G$.
\end{proof}

\begin{thm}\label{thm:characterize-prefactoriality}
	Let $C \subseteq \mathbb{P}^2(K)$ be an irreducible projective plane curve. Then $C$ is prefactorial if and only if it has degree at most $2$.
\end{thm}

\begin{proof}
	$(\Rightarrow)$ follows from Propositions \ref{prop:if-smooth-pref-then-deg} and \ref{prop:if-pref-then-smooth}. 
	
	$(\Leftarrow)$ If $C$ is a line, then for any $P \in C$ any other line through $C$ intersects $C$ only at $P$. If $C$ is a conic, then for any $P \in C$ the tangent line to $C$ at $P$ intersects $C$ only at $P$.
\end{proof}

\subsection*{Acknowledgements} The proof of Lemma \ref{lem:sets-of-pairs} was suggested by Marcello Mamino. We are grateful to Rita Pardini for suggesting the proof Proposition \ref{prop:if-smooth-pref-then-deg} and reading the Appendix. We thank Rosario Mennuni for helpful comments and Vincenzo Mantova for suggesting to use generalized Jacobians for the proof of Theorem \ref{thm:characterize-prefactoriality}. We also thank Marcus Tressl for discussing with us the problem of the interpretation of the complex field in the poset of points and curves. We are grateful to an anonymous referee for a very careful reading of the manuscript and for many comments leading to a substantial improvement of the paper. The authors were supported by the project PRIN 2022 ``Modelli, insiemi e classificazioni", prot. 2022TECZJA.

\end{document}